\documentclass[12pt]{amsart}
\usepackage{amsmath, amssymb}

\theoremstyle{plain}
\newtheorem{theorem}{Theorem}[section]
\newtheorem{proposition}[theorem]{Proposition}
\newtheorem{corollary}[theorem]{Corollary}
\newtheorem{lemma}[theorem]{Lemma}

\numberwithin{equation}{section}

\theoremstyle{definition}

\theoremstyle{remark}
\newtheorem*{remark}{Remark}
\newtheorem*{remarks}{Remarks}

\DeclareMathOperator{\Var}{Var}
\DeclareMathOperator{\trunc}{trunc}

\def \R {\mathbb{R}}

\def \E {\mathbb{E}}
\def \P {\mathbb{P}}
\def \one {{\bf 1}}
\def \NN {\mathcal{N}}

\def \e {\varepsilon}
\def \d {\delta}
\def \l {\lambda}
\def \s {\sigma}
\def \< {\langle}
\def \> {\rangle}
\def \^ {\widehat}

\def \Lip {{\rm{Lip}}}
\def \HS {{\rm{HS}}}

\begin{document}

\title[]{Spectral norm of products of random and deterministic matrices}

\author{Roman Vershynin}

\date{December 8, 2008; revised February 16, 2010}

\address{Department of Mathematics, University of Michigan, Ann Arbor, MI 48109, U.S.A.}
\email{romanv@umich.edu}

\thanks{Partially supported by NSF grant DMS FRG 0652617, 0918623 and Alfred P. Sloan Research Fellowship}

\begin{abstract}
  We study the spectral norm of matrices $W$ that can be factored as $W=BA$,
  where $A$ is a random matrix with independent mean zero entries and $B$ is a fixed matrix. 
  Under the $(4+\e)$-th moment assumption on the entries of $A$, we show that
  the spectral norm of such an $m \times n$ matrix $W$ is bounded by $\sqrt{m} + \sqrt{n}$,
  which is sharp.
  In other words, in regard to the spectral norm, products of random and deterministic 
  matrices behave similarly to random matrices with independent entries.
  This result along with the previous work of M.~Rudelson and the author implies that the 
  smallest singular value of a random $m \times n$ matrix with i.i.d. mean zero entries
  and bounded $(4+\e)$-th moment is bounded below by $\sqrt{m} - \sqrt{n-1}$ with high probability.
\end{abstract}

\maketitle

\section{Introduction}

This paper grew out of an attempt to understand the class of random matrices with non-independent 
entries, but which can be {\em factorized} through random matrices with independent entries.
Equivalently, we are interested in sample covariance matrices of a wide class of random vectors --
the linear transformations of vectors with independent entries. 

Here we study the spectral norm of such matrices.
Recall that the spectral norm $\|W\|$ is defined as the largest singular value of a matrix $W$,
which equals the largest eigenvalue of $\sqrt{W W^*}$.
Equivalently, the spectral norm can be defined as the $\ell_2 \to \ell_2$ operator norm:
$\|W\| = \sup_x \|Wx\|_2 / \|x\|_2$ where $\|\cdot\|_2$ denotes the Euclidean norm. 
The spectral norm of random matrices plays a notable role in particular in 
geometric functional analysis, computer science, statistical physics, and signal processing.

\subsection{Matrices with independent entries}

For random matrices with independent and identically distributed entries, the spectral norm 
is well studied.
Let $W$ be an $m \times n$ matrix whose entries are real independent and identically
distributed random variables with mean zero, variance $1$ and finite fourth moment. 
Estimates of the type
\begin{equation}					\label{optimal norm}
  \|W\| \sim \sqrt{n} + \sqrt{m}
\end{equation}
are known to hold (and are sharp) in both the limit regime for dimensions increasing to infinity,
and the non-limit regime where the dimensions are fixed. 
The meaning of \eqref{optimal norm} in the limit regime is that, for a family 
of matrices as above whose dimensions $m$ and $n$ increase to infinity and whose aspect ratio $m/n$ converges 
to a constant, the ratio $\|W\| / ( \sqrt{n} + \sqrt{m} )$ converges to $1$ almost surely \cite{YBK}.

In the non-limit regime, i.e. for arbitrary dimensions $n$ and $m$,
variants of \eqref{optimal norm} were proved by Y.~Seginer \cite{Se} and R.~Latala \cite{La}.
If $W$ is an $m \times n$ matrix whose entries are i.i.d. mean zero 
random variables, then denoting the rows of $W$ by $X_i$ and the columns by $Y_j$, 
the result of Y.~Seginer \cite{Se} states that
$$
\E \|W\| \le C \big( \E \max_i \|X_i\|_2  + \E \max_j \|Y_j\|_2 \big)
$$
where $C$ is an absolute constant.
This estimate is sharp because $\|W\|$ is obviously bounded below by the Euclidean norm of 
any row and any column of $W$.
Furthermore, if the entries $w_{ij}$ of the matrix $W$ are not necessarily identically distributed, 
then R.~Latala's result \cite{La} states that 
$$
\E \|W\| \le C \big( \max_i \E \|X_i\|_2  + \max_j \E \|Y_j\|_2 + \big( \sum_{i,j} \E w_{ij}^4 \big)^{1/4} \big).
$$
In particular, if $W$ is an $m \times n$ matrix whose entries are independent random variables
with mean zero and fourth moments bounded by $1$, then one can deduce from 
either Y.~Seginer's or R.~Latala's result that 
\begin{equation}					\label{non-asymptotic}
  \E \|W\| \le C ( \sqrt{n} + \sqrt{m} ).
\end{equation}
This is a variant of \eqref{optimal norm} in the non-limit regime.

The fourth moment hypothesis is known to be necessary. Consider again a
family of matrices whose dimensions $m$ and $n$ increase to infinity, and whose aspect ratio $m/n$ converges 
to a constant. If the entries are independent and identically distributed random variables 
with mean zero and infinite fourth moment, then the upper limit of the ratio $\|W\| / ( \sqrt{n} + \sqrt{m} )$ 
is infinite almost surely \cite{YBK}.

\subsection{The main result}

The main result of this paper is an extension of the optimal bound \eqref{non-asymptotic}
to the class of random matrices with non-independent entries, but which can be factored through 
a matrix with independent entries. 

\smallskip

\begin{theorem}                         \label{main}
  Let $\e \in (0,1)$ and let $m,n,N$ be positive integers.
  Consider a random $m \times n$ matrix $W = BA$,
  where $A$ is an $N \times n$ random matrix whose entries are independent random variables
  with mean zero and $(4+\e)$-th moment bounded by $1$, 
  and $B$ is an $m \times N$ non-random matrix such that $\|B\| \le 1$.
  Then
  \begin{equation}				\label{eq main}
    \E \|W\| \le C(\e) (\sqrt{n} + \sqrt{m})
  \end{equation}
  where $C(\e)$ is a function that depends only on $\e$.
\end{theorem}

\smallskip

\begin{remarks}
  {\bf 1.} An important feature of this result is that its conclusion is independent of the dimension $N$. 
  
  {\bf 2.} The proof of Theorem~\ref{main} yields the stronger estimate
  \begin{equation}						\label{B}
    \E \|W\| \le C(\e) (\|B\|\sqrt{n} + \|B\|_\HS)
  \end{equation}
  valid for arbitrary (non-random) $m \times N$ matrix $B$.
  This result is independent of the dimensions of the matrix $B$, 
  and therefore it holds for an arbitrary 
  linear operator $B$ acting from the $N$-dimensional Euclidean space 
  $\ell_2^N$ to an arbitrary Hilbert space.
     
  {\bf 3.} Theorem~\ref{main} can be interpreted in terms of {\em sample covariance matrices} of
  random vectors in $\R^m$ of the form $BX$, where $X$ is a random vector in $\R^N$ with independent entries.
  Indeed, let $A$ be the random matrix whose columns are $n$ independent samples of the vector $X$. 
  Then $W = BA$ is the matrix whose columns are $n$ independent samples of the random vector $BX$.
  The sample covariance matrix of the random vector $BX$ is defined as $\Sigma = \frac{1}{n} WW^*$.
  Theorem~\ref{main} states that the largest eigenvalue of $\Sigma$ is bounded by 
  $C_1(\e) (1 + m/n)$, which is further bounded by $C_2(\e)$ for the number of samples $n \gtrsim m$
  (and independently of the dimension $N$). 
  This problem was previously studied in \cite{BS 98}, \cite{BS 99} in the limit regime for $m=N$, 
  where the result must of course depend on $N$.
    
  {\bf 4.} Under the stronger subgaussian moment assumption \eqref{subgaussian} on the entries, 
  Theorem~\ref{main} is easy to prove using standard concentration and an $\e$-net argument.
  In contrast, if only some finite moment is assumed, we do not know any simple proof. 
\end{remarks}

\subsection{The smallest singular value}

Our main motivation for Theorem~\ref{main} was to complete the analysis of the 
{\em smallest singular value} of random rectangular matrices carried out by M.~Rudelson and the author
in \cite{RV}. The smallest singular value $s_{\min}(W)$ of a matrix $W$ can be equivalently 
described as $s_{\min}(W) = \inf_x \|Wx\|_2 / \|x\|_2$. 

Analyzing the smallest singular value is generally 
harder than analyzing the largest one (the spectral norm). 
The analogue of \eqref{optimal norm} for the smallest singular value of 
random $m \times n$ matrices $W$ (for $m > n$) is
\begin{equation}					\label{optimal smin}
  s_{\min}(W) \sim \sqrt{m} - \sqrt{n}.
\end{equation}
The optimal limit version of this result proved in \cite{BY} holds under exactly 
the same hypotheses as \eqref{optimal norm} --
for i.i.d. entries with mean zero, variance $1$ and finite fourth moment.

Many papers addressed \eqref{optimal smin} for fixed dimensions $n$, $m$.
Sufficiently tall matrices ($m \ge Cn$ for sufficiently large $C$) were studied in \cite{BDGJN};
extensions to genuinely rectangular matrices ($m > (1+\e)n$ for some $\e > 0$)
were studied in \cite{LPRT, AFMS, R 06}, with gradually improving dependence on $\e$.
An optimal version of \eqref{optimal smin} for all dimensions was obtained in \cite{RV}.
All these works put somewhat stronger moment assumptions than the fourth moment
of the entries $w_{ij}$ of the matrix $W$. A convenient assumption is that the entries $w_{ij}$ 
are {\em subgaussian} random variables. This means that all their moments 
are bounded by the corresponding moments of the standard normal random variable, i.e.
\begin{equation}						\label{subgaussian}
  (\E |w_{ij}|^p )^{1/p} \le M \sqrt{p} \qquad \text{for all } p \ge 1
\end{equation}
where $M$ is called the subgaussian moment. 
It was proved in \cite{RV} that if the entries of $W$ are i.i.d. mean zero subgaussian random 
variables with unit variance, then for every $t > 0$ one has
\begin{equation}										\label{for subgaussian}
\P \big( s_{\min}(W) \le t \big (\sqrt{m} - \sqrt{n-1} \big ) \big)
\le (C t)^{m-n+1} + e^{-cm}
\end{equation}
where $C, c > 0$ depend only on the subgaussian moment $M$.
In particular, for such matrices we have
\begin{equation}									\label{smin whp}
  s_{\min}(W) \ge c_1(\sqrt{m} - \sqrt{n-1})  \qquad \text{with high probability}
\end{equation}
where $c_1>0$ depends only on the desired probability and the subgaussian moment.
This result encompasses the case of square matrices where $m=n$ 
and hence \eqref{smin whp} yields $s_{\min}(W) \ge c_2/\sqrt{n}$. 
For Gaussian square matrices this optimal bound was obtained in \cite{E 88} and \cite{Sz}; 
for general square matrices a weaker bound $n^{-3/2}$ was obtained in \cite{R square}
and the best bound as above in \cite{RV square}; the estimate is shown to be optimal 
in \cite{RV CRAS}.

Whether \eqref{smin whp} holds under weaker moment assumptions 
was only known in the case of square matrices. 
It was proved in \cite{RV square} using \eqref{non-asymptotic} that \eqref{smin whp} holds
under the fourth moment assumption for square matrices, i.e. for $m=n$. 
Whether the same is true for arbitrary rectangular
matrices under the fourth moment assumption was left open in \cite{RV}. 
The bottleneck of the argument occurred in Proposition~7.3 on \cite{RV} where
we needed a correct bound on the spectral norm of a product of a random matrix and 
a fixed orthogonal projection. 
Such a bound was easy to get only under the subgaussian hypothesis. Theorem~\ref{main} of the 
present paper extends the argument of \cite{RV} for random matrices with bounded
$(4+\e)$-th moment. It follows directly from the argument of \cite{RV} and 
Theorem~\ref{main}.

\begin{corollary}[Smallest singular value]				\label{smin}
  Let $\e \in (0,1)$ and $m \ge n$ be positive integers.
  Let $A$ be a random $m \times n$ matrix whose entries are i.i.d. random variables 
  with mean zero, unit variance and $(4+\e)$-th moment bounded by $M$.
  Then, for every $\d > 0$ there exist $t > 0$ and $n_0$
  which depend only on $\e$, $\d$ and $M$, and such that
  $$
  \P \big( s_{\min}(A) \le t \big (\sqrt{m} - \sqrt{n-1} \big ) \big) 
  \le \d
  \qquad \text{for all } n \ge n_0.
  $$
\end{corollary}

This result follows by the argument in \cite{RV}, where one considers probability estimates
conditional on the event that the norm of a product $W$ of a random matrix
and a non-random orthogonal projection is small (see \cite[Proposition~7.3]{RV}).

\medskip

After this paper was written, two important related results appeared on the
universality of the smallest singular value in two extreme regimes --
for almost square matrices and for genuinely rectangular matrices. 
One of these results, 
by T.~Tao and V.~Vu \cite{TV universality}
works for square and almost square matrices where the the defect $m - n$ is constant.
It is valid for matrices with i.i.d. entries with mean zero, unit variance and bounded $C$-th moment 
where $C$ is a sufficiently large absolute constant. 
The result states that the smallest singular value of such $m \times n$ matrices $A$ is asymptotically 
the same as of the Gaussian matrix $G$ of the same dimensions and with i.i.d. standard
normal entries. Specifically, 
\begin{multline}
\P \big(m s_{\min}(G)^2 \le t - m^{-c} \big) - m^{-c}
\le \P \big(m s_{\min}(A)^2 \le t \big) \\
\le \P \big(m s_{\min}(G)^2 \le t + m^{-c} \big) + m^{-c}.
\end{multline}
This universality result, combined with the known asymptotic estimates 
of the smallest singular value of Gaussian matrices $s_{\min}(G)$
allows one to obtain bounds sharper than in Corollary~\ref{smin}. 
However, the universality result of \cite{TV universality} is only known in the almost square 
regime $m - n = O(1)$ (and under stronger moment assumptions), 
while Corollary~\ref{smin} is valid for all dimensions $m \ge n$.

Another recent universality result was obtained by O.~Feldheim and S.~Sodin \cite{FS}
for genuinely rectangular matrices, i.e. with aspect ratio $m/n$ separated from $1$ by a constant, 
and with subgaussian i.i.d. entries. In particular they proved the inequality
\begin{equation}										\label{FS}
\P \big(s_{\min}(A) \le (\sqrt{m}-\sqrt{n})^2 - t m \big)
\le \frac{C}{1-\sqrt{m/n}} \exp( -c n t^{3/2} ).
\end{equation}
Deviation inequalities \eqref{for subgaussian} and \eqref{FS} complement each other --
the former is multiplicative (and is valid for arbitrary dimensions) while the latter is 
additive (and is applicable for genuinely rectangular matrices). 
Each of these two inequalities clearly has the regime where it is stronger.

\medskip

\subsection{Outline of the argument}

Let us sketch the proof of Theorem~\ref{main}. We can assume 
that $m=n$ by adding an appropriate number of zero columns to $A$ or rows to $B$.
Since the columns of $A$ are independent, the columns $X_1, \ldots, X_n$ of the matrix $W$ 
are independent random vectors in $\R^n$. 
We would like to bound the spectral norm of 
$WW^* = \sum_j X_j \otimes X_j$, which is a sum of independent random operators.
For random vectors $X_j$ uniformly 
distributed in convex bodies, deviation inequalities for sums $\sum_j X_j \otimes X_j$
were studied in \cite{KLS, Bo, R 99, GHT, Pa, Au, ALPT}. 
For general distributions, a sharp estimate for such sums has been 
proved by M.~Rudelson \cite{R 99}.
This approach, which we develop in Section~\ref{s: rudelson}, leads us to the bound
\begin{equation}						\label{intro up to log}
  \E \|W\| \le C \sqrt{n \log n}.
\end{equation}
This bound is already independent of the dimension $N$, but is off by $\sqrt{\log n}$ 
from being optimal. The logarithmic term is unfortunately a limitation of this method.
This term comes from M.~Rudelson's result, Theorem~\ref{rudelson} below, where it is
needed in full generality. It would be useful to understand the situations where the logarithmic 
term can be removed from M.~Rudelson's theorem. 
So far, only one such situation is known from \cite{ALPT}
where the independent random vectors $X_j$ are uniformly distributed in a 
convex body.

In absence of a suitable variant of M.~Rudelson's theorem without the logarithmic term, 
the rest of our argument will proceed to remove this term from \eqref{intro up to log}
using the rich independence structure, which is inherited by the vectors $X_j$ 
from the random matrix $A$. However, the independence structure is encoded nontrivially via the linear
transformation $B$, which makes the entries of $X_j$ dependent).
A more delicate application of M.~Rudelson's theorem allows one to transfer the logarithmic 
term from the conclusion to the assumption. Namely, Theorem~\ref{small columns} 
establishes the optimal bound $\E \|W\| \le C \sqrt{n}$ in the case when all columns of 
$B$ are logarithmically small, i.e. their Euclidean norm is at most $\log^{-O(1)} n$.
While some columns of a general matrix $B$ may be large, the boundedness of $B$ implies that 
most columns are always logarithmically small -- all but all but $n \log^{O(1)} n$ of them.
So, we can remove from $B$ the already controlled small columns, which will make $B$ 
an almost square matrix. In other words, we can assume hereafter 
that $N = n \log^{O(1)} n$.

The advantage of almost square matrices is that the magnitude of their entries is easy to control.
A simple consequence of the $(4+\e)$-th moment hypothesis and Markov's inequality yields that 
the entries of $A = (a_{ij})$ satisfy $\max_{i,j} |a_{ij}| \le \sqrt{n}$ with high probability.
Note that the same estimate holds for square matrices ($N=n$) under the fourth moment assumption. 
So, in regard to the magnitude of entries, almost square matrices are similar
to exactly square matrices, for which the desired bound follows from R.~Latala's result \eqref{non-asymptotic}. 

This prompts us to construct the proof of Theorem~\ref{main} for almost square matrices
similarly to R.~Latala's argument in \cite{La}, i.e. using fairly standard concentration of measure
results in the Gauss space, coupled with delicate constructions of nets. 
We first decompose $A$ into a sum of matrices which contain entries of similar magnitude. 
As the magnitude increases, these matrices become sparser. This quickly reduces the problem
to random sparse matrices, whose entries are i.i.d. random variables valued in $\{-1,0,1\}$.
The spectral norm of random sparse matrices was studied in \cite{Kh} as a development 
of the work of Z.~Furedi and J.~Komlos \cite{FK}.
However, we need to bound the spectral norm of the matrix $W = BA$ rather than $A$. 
Independence of entries is not available for $W$, 
which makes it difficult to use the known combinatorial methods based 
on the bounding trace of high powers of $W$.

To summarize, at this point we have an almost square random sparse matrix $A$, and we need to bound
the spectral norm of $W = BA$, which is $\|W\| = \sup_x \|Wx\|_2$, where the supremum is over all unit
vectors $x \in \R^n$. The well known method is to first fix $x$ and bound $\|Wx\|_2$ with high probability; 
then take a union bound over all $x$ in a sufficiently fine net of the unit sphere of $\R^n$.
However, a probability bound for every fixed vector $x$, which follows from standard concentration inequalities,
is not strong enough to make this method work. 
{\em Sparse vectors} -- those which have few but large nonzero coordinates --
produce worse concentration bounds than {\em spread vectors}, 
which have many but small nonzero coordinates.
What helps us is that there are fewer sparse vectors on the sphere than there are spread vectors.
This leads to a tradeoff between concentration and entropy, i.e. between 
the probability with which $\|Wx\|_2$ is nicely bounded, and the size of a net for the 
vectors $x$ which achieve this probability bound.
One then divides the unit Euclidean sphere in $\R^n$ into classes of vectors according 
to their ``sparsity'', and uses the entropy-concentration tradeoff for each class separately.
This general line is already present in Latala's argument \cite{La}, and it was developed extensively
in the recent years, see e.g. \cite{LPRT, R square, RV square, RV}.
This argument is presented in Section~\ref{s: concentration}, 
where it leads to a useful estimate for norms of sparse matrices, Corollary~\ref{norm sparse matrices}.
With this in hand, one can quickly finish the proof of Theorem~\ref{main}.

\subsection*{Acknowledgement}
The author is grateful for the referee for careful reading of the manuscript,
and for many suggestions which greatly improved the presentation.

\section{Preliminaries}                         \label{s: prelims}

\subsection{Notation}

Throughout the paper, the results are stated and proved over the field of real numbers. 
They are easy to generalize to complex numbers.

We denote by $C, C_1, c, c_1 \ldots$ positive absolute constants,
and by $C(\e), C_1(\e), \ldots$ positive quantities that may depend only on the parameter $\e$.
Their values can change from line to line.

The standard inner product in $\R^n$ is denoted $\< x, y \> $. For a vector $x \in \R^n$, we
denote the cardinality of its support by $\|x\|_0 = |\{ j : x_j \ne 0\}|$, the Euclidean norm by 
$\|x\|_2 = ( \sum_j x_j^2 )^{1/2}$, and the sup-norm by $\|x\|_\infty = \max_j |x_j|$.
The unit Euclidean ball in $\R^n$ is denoted by $B_2^n = \{ x : \|x\|_2 \le 1\}$, and 
the unit Euclidean sphere in $\R^n$ is denoted by $S^{n-1} = \{ x : \|x\|_2 = 1\}$.

The tensor product of vectors $x, y \in \R^n$ is the linear operator $x \otimes y$ 
on $\R^n$ defined as $(x \otimes y)(z) = \< x, z\> y$ for $z \in \R^n$.

\subsection{Concentration of measure}

The method that we carry out in Section~\ref{s: concentration} uses concentration in the Gauss space
in combination with constructions of $\e$-nets. Here we recall some basic facts we need.

The standard Gaussian random vector $g \in \R^m$ is a random vector whose coordinates are
independent standard normal random variables. The following concentration inequality can be found
e.g. in \cite[inequality~(1.5)]{LT}.

\begin{theorem}[Gaussian concentration]         \label{gauss concentration}
  Let $f : \R^m \to \R$ be a Lipschitz function.
  Let $g$ be a standard Gaussian random vector in $\R^m$.
  Then for every $t > 0$ one has
  $$
  \P ( f(g) - \E f(g) > t ) \le \exp(-c_0 t^2 / \|f\|_\Lip^2) 
  $$
  where $c_0 \in (0,1)$ is an absolute constant.
\end{theorem}

As a very restrictive but useful example, Theorem~\ref{gauss concentration} implies
the following deviation inequality for sums of independent exponential random variables $g_i^2$
(which can also be derived by the more standard approach via moment generating functions).

\begin{corollary}[Sums of exponential random variables]      \label{exponential deviation}
  Let $d = (d_1,\ldots,d_m)$ be a vector of real numbers,
  and let $g_1,\ldots,g_m$ be independent standard normal random variables.
  Then, for every $t >0$ we have
  $$
  \P \Big\{ \Big( \sum_{i=1}^m d_i^2 g_i^2 \Big)^{1/2} > \|d\|_2 + t \Big\}
  \le \exp(-c_0 t^2 / \|d\|_\infty^2).
  $$
\end{corollary}

\begin{proof}
The function $f(y) =  (\sum_{i=1}^m d_i^2 y_i^2)^{1/2}$ is a Lipschitz function on $\R^m$
with $\|f\|_\Lip = \|d\|_\infty$. Moreover, H\"older's inequality implies that
$$
\E f(g) = \E \Big( \sum_{i=1}^m d_i^2 g_i^2 \Big)^{1/2}
\le \Big( \E \sum_{i=1}^m d_i^2 g_i^2 \Big)^{1/2}
= \|d\|_2.
$$
Theorem~\ref{gauss concentration} completes the proof.
\end{proof}

\medskip

Another classical deviation inequality we will need is Bennett's inequality,
see e.g. \cite[Theorem~2]{BBL}:

\begin{theorem}[Bennett's inequality]           \label{bennett}
  Let $X_1,\ldots,X_N$ be independent mean zero random variables 
  such that $|X_i| \le 1$ for all $i$.
  Consider the sum $S = X_1 + \cdots + X_N$ and let $\s^2 := \Var(S)$.
  Then, for every $t > 0$ we have
  $$
  \P (S > t) \le \exp \big( -\s^2 h(t/\s^2) \big)
  $$
  where $h(u) = (1+u) \log(1+u) - u$.
\end{theorem}

\medskip

We will also need M.~Talagrand's concentration inequality for convex Lipschitz funcitons
from \cite[Theorem~6.6]{Ta}; see also \cite[Corollary~4.10]{Le} and the discussion below it. 

\begin{theorem}[Concentration of Lipschitz convex functions]								\label{talagrand}
  Let $X_1,\ldots,X_m$ be independent random variables such that 
  $|X_i| \le K$ for all $i$.
  Let $f : \R^m \to \R$ be a convex and $1$-Lipschitz function. 
  Then for every $t > 0$ one has
  $$
  \P \big( |f(X_1,\ldots,X_m) - \E f(X_1,\ldots,X_m)| > K t \big) \le 4 \exp(-t^2/4).
  $$
\end{theorem}

\subsection{Nets}

Consider a subset $U$ of a normed space $X$, and let $\e > 0$.
Recall that an {\em $\e$-net} of $U$ is a subset $\NN$ of $U$ such that
the distance from any point of $U$ to $\NN$ is at most $\e$.
In other words, for every $x \in U$ there exists $y \in \NN$ such that $\|x-y\|_X \le \e$.

The following estimate follows by a volumetric argument, see e.g. the proof of 
Lemma~9.5 in \cite{LT}.

\begin{lemma}[Cardinality of $\e$-nets]           \label{net cardinality}
  Let $\e \in (0,1)$. The unit Euclidean ball $B_2^n$ and the unit Euclidean sphere
  $S^{n-1}$ in $\R^n$ both have $\e$-nets of cardinality at most $(1 + 2/\e)^n$.
\end{lemma}

When computing norms of linear operators, $\e$-nets provide a convenient discretization
of the problem. We formalize it in the next proposition.

\begin{proposition}[Computing norms on nets]     \label{norm on nets}
  Let $A : X \to Y$ be a linear operator between normed spaces $X$ and $Y$,
  and let $\NN$ be an $\e$-net of either the unit sphere $S(X)$ or the unit ball $B(X)$ of $X$ 
  for some $\e \in (0,1)$.
  Then
  $$
  \|A\| \le \frac{1}{1-\e} \sup_{x \in \NN} \|Ax\|_Y.
  $$
\end{proposition}

\begin{proof}
We give the proof for an $\e$-net of the unit sphere; the case of the unit ball is similar.
Every $z \in S(X)$ has the form $z = x + h$, where $x \in \NN$ and $\|h\|_X \le \e$.
Since $\|A\| = \sup_{z \in S(X)} \|Az\|_Y$, the triangle inequality yields
$$
\|A\| \le \sup_{x \in \NN} \|Ax\|_Y + \sup_{\|h\|_X \le \e} \|Ah\|_Y.
$$
The last term in the right hand side is bounded by $\e \|A\|$. Thus we have
shown that
$$
(1-\e)\|A\| \le \sup_{x \in \NN} \|Ax\|_Y.
$$
This completes the proof.
\end{proof}

\subsection{Symmetrization}

We will use the standard symmetrization technique as was done in \cite{La};
see more general inequalities in e.g. \cite[Section 6.1]{LT}.
To this end, let the matrices $A = (a_{ij})$ and $B$ be as in Theorem~\ref{main}.
Let $A' =(a'_{ij})$ be an independent copy of $A$, and let $\e_{ij}$ be independent symmetric
Bernoulli random variables.
Then, by Jensen's inequality,
\begin{align*}
\E \|BA\|
  &= \E \|B(A - \E A')\|
  \le \E \|B(A - A')\| \\
  &= \E \|B(\e_{ij} (a_{ij} - a'_{ij}))\|
  \le 2 \E \|B(\e_{ij} a_{ij})\|.
\end{align*}
Therefore, we can assume without loss of generality in Theorem~\ref{main}
that $a_{ij}$ are symmetric random variables.
Furthermore, let $g_{ij}$ be independent standard normal random variables.
Then, again by Jensen's inequality,
\begin{align*}
\E \|B(g_{ij} a_{ij})\|
  &= \E \|B(\e_{ij} |g_{ij}| a_{ij})\|
  \ge \E \|B(\e_{ij} \E(|g_{ij}|) a_{ij})\| \\
  &= (2/\pi)^{1/2} \E \|B(\e_{ij} a_{ij})\|.
\end{align*}
Therefore
\begin{equation}                                \label{making gaussian}
  \E \|BA\| \le  (2\pi)^{1/2} \E \|B(g_{ij} a_{ij})\|.
\end{equation}
Conditioning on $a_{ij}$, we thus reduce the problem to random {\em gaussian} matrices.

\medskip

We will use a similar symmetrization technique several times in our argument. 
In particular, in the proof of Lemma~\ref{max tensor sums}
we apply the following observation, which can be deduced from standard
symmetrization lemma (\cite{LT} Lemma~6.3) and the contraction principle (\cite{LT} Theorem~4.4).
For the reader's convenience we include a direct proof.

\begin{lemma}[Symmetrization]					\label{symmetrization}
  Consider independent mean zero random variables $Z_{ij}$ 
  such that $|Z_{ij}| \le 1$, independent symmetric Bernoulli random variables $\e_{ij}$,
  and vectors $x_{ij}$ in some Banach space, where both $i$ and $j$ range 
  in some finite index sets.
  Then 
  $$
  \E \max_j \Big\| \sum_i Z_{ij} x_{ij} \Big\| 
  \le 2 \E \max_j \Big\| \sum_i \e_{ij} x_{ij} \Big\|.
  $$  
\end{lemma}

\begin{proof}
To be specific, we can assume that both indices $i$ and $j$ range in the interval 
$\{1,\ldots,n\}$ for some integer $n$.
Let $(Z'_{ij})$ denote an independent copy of the sequence of random variables $(Z_{ij})$.
Then $Z_{ij} - Z'_{ij}$ are symmetric random variables. We have
\begin{align*}
\E \max_j \Big\| \sum_i Z_{ij} x_{ij} \Big\|
  &\le  \E \max_j \Big\| \sum_i (Z_{ij} - \E Z'_{ij}) x_{ij} \Big\|	\quad \text{(since $\E Z'_{ij} = 0$)} \\
  &\le  \E \max_j \Big\| \sum_i (Z_{ij} - Z'_{ij}) x_{ij} \Big\|	\quad \text{(by Jensen's inequality)} \\
  &=  \E \max_j \Big\| \sum_i \e_{ij} (Z_{ij} - Z'_{ij}) x_{ij} \Big\|	\quad \text{(by symmetry)} \\
  &\le 2 \max_{|a_{ij}| \le 1} \E \max_j \Big\| \sum_i \e_{ij} a_{ij} x_{ij} \Big\|
\end{align*}
where the last line follows because $|Z_{ij} - Z'_{ij}| \le |Z_{ij}| + |Z'_{ij}| \le 2$.
The function on $\R^{n^2}$
$$
(a_{ij})_{i,j=1}^n \mapsto \E \max_j \Big\| \sum_i \e_{ij} a_{ij} x_{ij} \Big\|
$$
is a convex function. Therefore, on the compact convex set $[-1,1]^{n^2}$ it attains its 
maximum on the extreme points, where all $a_{ij} = \pm 1$. By symmetry, the function 
takes the same value at each extreme point, which equals 
$$
\E \max_j \Big\| \sum_i \e_{ij} x_{ij} \Big\|.
$$
This completes the proof.
\end{proof}

\subsection{Truncation and conditioning}

We will need some elementary observations related to truncation and conditioning of random variables.

\begin{lemma}[Truncation]               \label{truncation}
  Let $X$ be a non-negative random variable, and let $M > 0$, $p \ge 1$.
  Then
  $$
  \E X \one_{\{X \ge M\}} \le \frac{\E X^p}{M^{p-1}}.
  $$
\end{lemma}

\begin{proof}
Indeed,
$$
\E X \one_{\{X \ge M\}}
\le \E X (X/M)^{p-1} \one_{\{X \ge M\}}
\le \E X^p / M^{p-1}.
$$
The Lemma is proved.
\end{proof}

\medskip

We will also need two elementary conditioning lemmas. In Section~\ref{s: concentration},
we will need to control the maximal magnitude of the entries $M_0 = \max_{ij} |a_{ij}|$ 
of the random matrix $A$. Conditioning on $M_0$ will unfortunately destroy the independence of the entries. 
So, we will instead condition on an event $\{ M_0 \le t \}$ for fixed $t$, which will clearly
preserve the independence. This conditional argument used in the proof of Corollary~\ref{almost square}
relies on the following two elementary lemmas.

\begin{lemma}					\label{exp reduces by conditioning}
  Let $X$ be a random variable and $K$ be a real number.
  Then 
  $$
  \E (X \,|\, X \le K) \le \E X.
  $$
\end{lemma}

\begin{proof}
By the law of total probability,
$$
\E X = \E (X \,|\, X \le K) \, \P(X \le K) +  \E (X \,|\, X > K) \, \P(X > K).
$$
Thus $\E X$ is a convex combination of the numbers 
$a = \E (X \,|\, X \le K)$ and $b = \E (X \,|\, X > K)$.
Since clearly $a \le K \le b$, we must have $a \le \E X \le b$.
\end{proof}

\begin{lemma}					\label{conditional exp}
  Let $X$, $Y$ be non-negative random variables. 
  Assume there exists $K, L > 0$ such that one has for every $t \ge 1$:
  \begin{equation}					\label{KL}
  \E (X^2 \,|\, Y \le t) \le K^2 t, \qquad
  \P (Y > L t) \le \frac{1}{t^2}.
  \end{equation}
  Then $\E X \le C K \sqrt{L}$.
\end{lemma}

\begin{proof}
Without loss of generality we can assume that $K = 1$ by rescaling $X$ to $X/K$.
Thus we have for every $t \ge 1$: 
\begin{equation}							\label{after scaling}
  \E X^2 \one_{\{ Y \le t \}} \le \E (X^2 \,|\, Y \le t) \le t.,
\end{equation}
We consider the decomposition
$$
\E X = \E X \one_{\{ Y \le L \}} 
  + \sum_{k=1}^\infty  \E X \one_{\{ 2^{k-1} L < Y \le 2^k L \}}.
$$
By \eqref{after scaling} and H\"older's inequality, the first term is bounded as
$$
\E X \one_{\{ Y \le L \}} 
\le \big( \E X^2 \one_{\{ Y \le L \}} \big)^{1/2}
\le \sqrt{L}.
$$
Further terms can be estimated by Cauchy-Schwarz inequality and using \eqref{after scaling}
and the second inequality in \eqref{KL}. Indeed, 
\begin{align*}
\E X \one_{\{ 2^{k-1} L < Y \le 2^k L \}}
  &= \E X \one_{\{ Y \le 2^k L \}} \one_{\{ Y > 2^{k-1} L \}} \\
  &\le \big( \E X^2 \one_{\{ Y \le 2^k L \}} \big)^{1/2}
    \big( \P \{ Y > 2^{k-1} L \} \big)^{1/2}\\
  &\le (2^k L)^{1/2} \cdot \frac{1}{2^{k-1}}
  = \sqrt{L} \, 2^{1-k/2}.
\end{align*}
Therefore 
$$
\E X \le \sqrt{L} + \sum_{k=1}^\infty \sqrt{L} \, 2^{1-k/2} \le C \sqrt{L}.
$$
This completes the proof.
\end{proof}

\subsection{On the deterministic matrix $B$ in Theorem~\ref{main}.} \label{s: reductions}

We start with two initial observations that will make our proof of Theorem~\ref{main}
more transparent.
By adding an appropriate number of zero rows to $B$ or zero columns to $A$ we can 
assume without loss of generality that $n=m$, thus $B$ is an $n \times N$ matrix.

Throughout the proof of Theorem~\ref{main}, we shall denote the columns of 
such a matrix $B$ by $B_1,\ldots,B_N$. They are non-random vectors in $\R^n$, which satisfy
\begin{equation}                        \label{B HS}
\max_i \|B_i\|_2 \le \|B\| \le 1; \qquad
\sum_{i=1}^N \|B_i\|_2^2
= \|B\|_{\HS}^2
\le n \|B\|
\le n
\end{equation}
where $\|\cdot\|_\HS$ denotes the Hilbert-Schmidt norm.
Throughout the argument, we will only have access to the matrix $B$ through 
inequalities \eqref{B HS}. This explains Remark~2 following Theorem~\ref{main}, 
which states that the range space of $B$ is irrelevant as long as we control 
the spectral and Hilbert-Schmidt norms of $B$.

\section{Approach via M.~Rudelson's theorem}                \label{s: rudelson}

\subsection{M.~Rudelson's theorem}

Our first approach, which will yield Theorem~\ref{main} up to a logarithmic factor, 
rests on the following result.
Here and thereafter, by $\e_1, \e_2, \ldots$ we denote independent symmetric 
Bernoulli random variables, i.e. independent random variables such that 
$\P (\e_i = \pm 1) = 1/2$.

\begin{theorem}[M.~Rudelson \cite{R 99}]                        \label{rudelson}
  Let $u_1,\ldots,u_M$ be vectors in $\R^m$.
  Then, for every $p \ge 1$, one has
  $$
  \Big( \E \Big\| \sum_{i=1}^M \e_i u_i \otimes u_i \Big\|^p \Big)^{1/p}
  \le C (\sqrt{p} + \sqrt{\log m})
  \cdot \max_i \|u_i\|_2
  \cdot \Big\| \sum_{i=1}^M u_i \otimes u_i \Big\|^{1/2}.
  $$
  In particular, for every $t > 0$, with probability at least $1 - 2 m e^{-ct^2}$ one has
  $$
  \Big\| \sum_{i=1}^M \e_i u_i \otimes u_i \Big\|
  \le t
  \cdot \max_i \|u_i\|_2
  \cdot \Big\| \sum_{i=1}^M u_i \otimes u_i \Big\|^{1/2}.
  $$
\end{theorem}

The first estimate is taken from \cite[inequality (3.4)]{R 99}. The second estimate can be easily derived from it
using the following elementary lemma:

\begin{lemma}[Moments and tails]                   \label{int-tail}
  Suppose a non-negative random variable $X$ satisfies for some $m \ge 1$ that
  $$
  (\E X^p)^{1/p} \le \sqrt{p} + \sqrt{\log m} \quad \text{for every } p \ge 1.
  $$
  Then
  $$
  \P (X \ge t) \le 2m e^{-ct^2} \quad \text{for every } t > 0.
  $$
\end{lemma}

\begin{proof}
Suppose first that $t \ge \max(1, \sqrt{\log m})$.
Let $p := t^2$. Then $\sqrt{p} \ge \sqrt{\log m}$, so the hypothesis gives
$(\E X^p)^{1/p} \le 2 \sqrt{p}$.
By Markov's inequality,
$$
\P (X \ge 2e t) = \P (X^p \ge (2e t)^p)
\le \frac{(2 \sqrt{p})^p}{(2et)^p}
= e^{-t^2}.
$$
Next, if $t < \max(1, \sqrt{\log m})$ then by choosing the absolute constant $c>0$ 
sufficiently small  right hand side of \eqref{X ge t} is larger than $1$
for a sufficiently small absolute constant $c$ .
Therefore, for every $t > 0$ one has
\begin{equation}                        \label{X ge t}
  \P (X \ge 2et) \le 2m e^{-t^2/2}
\end{equation}
because if $t < \max(1, \sqrt{\log m})$ then the right hand side of \eqref{X ge t} is larger than one, 
which makes the inequality trivial.
This completes the proof.
\end{proof}

\medskip

The next lemma is a consequence of M.~Rudelson's Theorem~\ref{rudelson}
and a standard symmetrization argument.

\begin{lemma}                               \label{rudelson arbitrary}
  Let $X_1,\ldots,X_n$ be independent random vectors in $\R^m$ such that
  \begin{equation}                          \label{Xj tensor Xj}
    \| \E X_j \otimes X_j \| \le 1 \qquad \text{for every } j.
  \end{equation}
  Then
  $$
  \E \Big\| \sum_{j=1}^n X_j \otimes X_j \Big\|
  \le C n + C \log(2m) \, \E \max_j \|X_j\|_2^2.
  $$
\end{lemma}

\begin{proof}
Let $\e_1,\ldots,\e_n$ be independent symmetric Bernoulli random variables.
By the triangle inequality, the standard symmetrization argument (see e.g. \cite[Lemma~6.3]{LT}),
and the assumption, we have
\begin{align*}
  E
  &:= \E \Big\| \sum_{j=1}^n X_j \otimes X_j \Big\|
  \le \E \Big\| \sum_{j=1}^n (X_j \otimes X_j - \E X_j \otimes X_j) \Big\|
    + \Big\| \sum_{j=1}^n \E X_j \otimes X_j \Big\| \\
  &\le 2 \E \Big\| \sum_{j=1}^n \e_j X_j \otimes X_j \Big\| + n.
\end{align*}
Condition on the random variables $X_1,\ldots,X_n$, and apply Theorem~\ref{rudelson}.
Writing $\E_\e$ to denote the conditional expectation (i.e. the expectation
with respect to the random variables $\e_1,\ldots,\e_n$), we have
$$
\E_\e \Big\| \sum_{j=1}^n \e_j X_j \otimes X_j \Big\|
\le C \sqrt{\log(2m)} \cdot \max_j \|X_j\|_2 \cdot \Big\| \sum_{j=1}^n X_j \otimes X_j \Big\|^{1/2}.
$$
Now we take expectation with respect to $X_1,\ldots,X_n$ and use Cauchy-Schwarz inequality to get
$$
E \le C \sqrt{\log(2m)} \cdot \big( \E \max_j \|X_j\|_2^2 \big)^{1/2} \cdot E^{1/2} + n.
$$
The conclusion of the lemma follows.
\end{proof}

\subsection{Theorem~\ref{main} up to a logarithmic term}

We now state a version of Theorem~\ref{main} with a logarithmic factor.

\begin{proposition}                         \label{up to log}
  Let $N,n$ be positive integers. 
  Consider an $N \times n$ random matrix $A$ whose entries 
  are independent random variables
  with mean zero and $4$-th moment bounded by $1$.
  Let $B$ be an $n \times N$ matrix such that $\|B\| \le 1$.
  Then
  $$
  \E \|BA\| \le C \sqrt{n \log (2n)}.
  $$
\end{proposition}

The proof will need two auxiliary lemmas. Recall that $B_1,\ldots,B_N$ denote the columns of the matrix $B$.

\begin{lemma}                               \label{exp var}
  Let $a_1,\ldots,a_N$ be independent random variables 
  with mean zero and $4$-th moment bounded by $1$.
  Consider the random vector $X$ in $\R^n$ defined as
  $$
  X = \sum_{i=1}^N a_i B_i.
  $$
  Then
  $$
  \E \|X\|_2^2 \le n, \qquad
  \Var(\|X\|_2^2) \le 3 n.
  $$
\end{lemma}

\begin{proof}
The estimate on the expectation follows easily from \eqref{B HS}:
\begin{equation}                                \label{E X square}
  \E \|X\|_2^2 = \sum_{i=1}^N \E(a_i^2) \|B_i\|_2^2
  \le \sum_{i=1}^N \|B_i\|_2^2
  \le n.
\end{equation}
To estimate the variance, we need to compute
$$
\E \|X\|_2^4
= \E \< X, X\> ^2
= \sum_{i,j,k,l=1}^N \E (a_i a_j a_k a_l) \< B_i, B_j \> \< B_k, B_l \> .
$$
By independence and the mean zero assumption, 
the only nonzero terms in this sum are those for which $i=j; k=l$ or
$i=k; j=l$ or $i=l; j=k$.
Therefore
\begin{align*}
  \E \|X\|_2^4
  &= \sum_{i,j=1}^N \E(a_i^2 a_j^2) \|B_i\|_2^2 \|B_j\|_2^2
    + 2 \sum_{i,j=1}^N \E(a_i^2 a_j^2) \< B_i, B_j\> ^2 \\
  &=\sum_{i=1}^N \E(a_i^4) \|B_i\|_2^4
  +\sum_{\substack{i,j=1 \\ i \ne j}}^N \E(a_i^2) \E(a_j^2) \|B_i\|_2^2 \|B_j\|_2^2
  + 2 \sum_{i,j=1}^N \E(a_i^2 a_j^2) \< B_i, B_j\> ^2 \\
  &=: I_1 + I_2 + I_3.
\end{align*}
By the fourth moment assumption and using \eqref{B HS} we have
$$
I_1 \le \sum_{i=1}^N \|B_i\|_2^4
\le \max_{i} (\|B_i\|_2^2) \sum_{i=1}^N \|B_i\|_2^2
\le n
$$
Squaring the sum in \eqref{E X square}, we see that
$$
I_2 \le (\E \|X\|_2^2)^2.
$$
Finally, since by Cauchy-Schwarz inequality
$\E(a_i^2 a_j^2) \le \sqrt{\E(a_i^4) \E(a_j^4)} \le 1$, and
using \eqref{B HS} again, we obtain
$$
I_3
\le 2 \sum_{i,j=1}^N \< B_i, B_j\> ^2
= 2\|B^* B\|_\HS^2
\le 2\|B^*\|^2 \|B\|_\HS^2
= 2\|B\|^2 \|B\|_\HS^2
\le 2 n.
$$
Putting all this together, we obtain
$$
\Var(\|X\|_2^2)
= \E \|X\|_2^4 - (\E \|X\|_2^2)^2
\le I_1 + I_3 \le 3 n.
$$
This completes the proof.
\end{proof}

\begin{lemma}                   \label{columns BA}
  Let $A$ and $B$ be matrices as in Proposition~\ref{up to log}.
  Let $X_1,\ldots,X_n \in \R^n$ denote the columns of the matrix $BA$.
  Then
  $$
  \E \max_{j=1,\ldots,n} \|X_j\|_2^2 \le C n.
  $$
\end{lemma}

\begin{remark}
  This result says that all columns of the matrix $BA$ have norm $O(\sqrt{n})$
  with high probability. Since the spectral norm of a matrix is bounded
  below by the norm of any column, this result is a necessary step in
  proving our desired estimate $\|BA\| = O(\sqrt{n})$.
\end{remark}

\begin{proof}
Let, as usual, $B_1,\ldots,B_N \in \R^n$ denote the columns of the matrix $B$,
and let $a_{ij}$ denote the entries of the matrix $A$.
Then
\begin{equation}                            \label{Xj}
  X_j = \sum_{i=1}^N a_{ij} B_i, \qquad j=1,\ldots,n.
\end{equation}
Let us fix $j \in \{1,\ldots,n\}$ and use Lemma~\ref{exp var}. This gives
\begin{equation}                            \label{E Xj Var Xj}
  \E \|X_j\|_2^2 \le n, \qquad
  \Var(\|X_j\|_2^2) \le 3 n.
\end{equation}
Now we use Chebychev's inequality, which states that for a random variable
$Z$ with $\s^2 = \Var(Z)$ and for an arbitrary $k > 0$, one has
$$
\P( |Z - \E Z| > k \s) \le \frac{1}{k^2}.
$$
Let $t > 0$ be arbitrary.
Using Chebychev's inequality along with \eqref{E Xj Var Xj}
for $Z = \|X_j\|_2^2$, $k = t \sqrt{n}$, we obtain
$$
\P \big( \|X_j\|_2^2 > (1+\sqrt{3} \, t) n \big) \le \frac{1}{t^2 n}.
$$
Taking the union bound over all $j=1,\ldots,n$, we conclude that
$$
\P \big( \max_{j=1,\ldots,n} \|X_j\|_2^2 > (1+\sqrt{3} \, t) n \big)
\le n \cdot \frac{1}{t^2 n} = \frac{1}{t^2}.
$$
Integration completes the proof.
\end{proof}

\medskip

\begin{proof}[Proof of Proposition~\ref{up to log}.]
Let $X_1,\ldots,X_n \in \R^n$ denote the columns of the matrix $BA$.
We are going to apply Lemma~\ref{rudelson arbitrary}. 
In order to check that condition \eqref{Xj tensor Xj} holds,
we consider an arbitrary vector $x \in S^{n-1}$ and use
representation \eqref{Xj} to compute
\begin{align*}
  \E \< X_j, x \> ^2
  &= \E \Big( \sum_{i=1}^N a_{ij} \< B_i, x \> \Big)^2
  = \sum_{i=1}^N \E(a_{ij}^2) \< B_i, x \> ^2
  \le \sum_{i=1}^N \< B_i, x \> ^2 \\
  &= \|B^* x\|_2^2
  \le \|B^*\|^2
  = \|B\|^2 \le 1.
\end{align*}
This shows that condition \eqref{Xj tensor Xj} holds. Lemma~\ref{rudelson arbitrary} then gives
$$
\E \|BA\|^2
= \E \Big\| \sum_{j=1}^n X_j \otimes X_j \Big\|
\le C n + C \log(2n) \, \E \max_{j=1,\ldots,n} \|X_j\|_2^2.
$$
Estimating the maximum in the right hand side using Lemma~\ref{columns BA},
we conclude that
$$
\E \|BA\|^2 \le C_1 n \log(2n).
$$
This completes the proof.
\end{proof}

\subsection{Tradeoff between the matrix norm and the magnitude of entries}

We would like now to gain more control over the logarithmic factor than we have
in Proposition~\ref{up to log}.
Our next result establishes a tradeoff between the logarithmic factor and the
magnitude of the matrices $A$, $B$. It will be used in the proof of Theorem~\ref{small columns}.

\begin{proposition}                 \label{controlled columns}
  Let $a, b \ge 0$ and $N,n$ be positive integers.
  Let $A$ be an $N \times n$ matrix whose entries are random independent variables $a_{ij}$ with mean zero
  and such that
  $$
  \E a_{ij}^2 \le 1, \qquad |a_{ij}| \le a \qquad \text{for every } i,j.
  $$
  Let $B$ be an $n \times N$ matrix such that $\|B\| \le 1$, and whose columns satisfy
  $$
  \|B_i\|_2 \le b \qquad \text{for every } i.
  $$
  Then
  $$
  \E \|BA\| \le C (1 + a b^{1/2} \log^{1/4}(2n)) \sqrt{n}.
  $$
\end{proposition}

The proof will again be based on M.~Rudelson's Theorem~\ref{rudelson},
although this time we use Rudelson's theorem in a more delicate way:

\begin{lemma}                       \label{max tensor sums}
  Under the assumptions of Proposition~\ref{controlled columns}, we have
  $$
  \E \max_{j=1,\ldots,n} \Big\| \sum_{i=1}^N a_{ij}^2 B_i \otimes B_i \Big\|
  \le C (1 + a^2 b \sqrt{\log(2n)}).
  $$
\end{lemma}

\begin{proof}
Fix $j \in \{1,\ldots,n\}$.
Let $\mu_{ij}^2 := \E a_{ij}^2$.
By the triangle inequality,
\begin{equation}                                \label{making mean zero}
  \Big\| \sum_{i=1}^N a_{ij}^2 B_i \otimes B_i \Big\|
  \le  \Big\| \sum_{i=1}^N (a_{ij}^2 - \mu_{ij}^2) B_i \otimes B_i \Big\|
    + \Big\| \sum_{i=1}^N \mu_{ij}^2 B_i \otimes B_i \Big\|.
\end{equation}
Since $0 \le \mu_{ij}^2 \le 1$ and
\begin{equation}                                \label{Bj tensor Bj}
  \Big\| \sum_{i=1}^N B_i \otimes B_i \Big\| \le \|B\|^2 \le 1,
\end{equation}
we have
\begin{equation}                                \label{residual}
  \Big\| \sum_{i=1}^N \mu_{ij}^2 B_i \otimes B_i \Big\| 
  \le \Big\| \sum_{i=1}^N B_i \otimes B_i \Big\|
  \le 1.
\end{equation}
Next, clearly $\mu_{ij}^2 \le a^2$, so
$$
\E (a_{ij}^2 - \mu_{ij}^2) = 0, \qquad |a_{ij}^2 - \mu_{ij}^2| \le 2a^2.
$$
Symmetrization Lemma~\ref{symmetrization} yields 
\begin{equation}                                \label{symmetrized}
  \E \max_{j=1,\ldots,n} \Big\| \sum_{i=1}^N (a_{ij}^2 - \mu_{ij}^2) B_i \otimes B_i \Big\|
  \le 2 a^2 \, \E \max_{j=1,\ldots,n} \Big\| \sum_{i=1}^N \e_{ij} B_i \otimes B_i \Big\|
\end{equation}
where $\e_{ij}$ denote independent symmetric Bernoulli random variables.

Let $t > 0$. By the second part of M.~Rudelson's Theorem~\ref{rudelson}
and taking the union bound over $n$ random variables, we conclude that,
with probability at least $1 - 2n^2 e^{-ct^2}$, we have
$$
\max_{j=1,\ldots,n} \Big\| \sum_{i=1}^N \e_{ij} B_i \otimes B_i \Big\|
\le t \cdot \max_{i=1,\ldots,N} \|B_i\|_2 \cdot \Big\| \sum_{i=1}^N B_i \otimes B_i \Big\|^{1/2}
\le t b
$$
The second estimate follows from \eqref{Bj tensor Bj} and since
$\max_i \|B_i\|_2 \le b$ by the hypothesis.

Let $s > 0$ be arbitrary. We apply the above estimate for $t$ chosen so that
$2n^2 e^{-ct^2} = e^{-s^2}$. This shows that, with probability at least $1 - e^{-s^2}$,
one has
$$
\max_{j=1,\ldots,n} \Big\| \sum_{i=1}^N \e_{ij} B_i \otimes B_i \Big\|
\le t b \le C_1 b (\sqrt{\log(2n)} + s).
$$
Integration implies that
$$
\E \max_{j=1,\ldots,n} \Big\| \sum_{i=1}^N \e_{ij} B_i \otimes B_i \Big\|
\le C_2 b \sqrt{\log(2n)}.
$$
Putting this into \eqref{symmetrized} and, together with \eqref{residual},
back into \eqref{making mean zero}, we complete the proof.
\end{proof}

\medskip

\begin{proof}[Proof of Proposition~\ref{controlled columns}.]
By the symmetrization argument (see \eqref{making gaussian}), we can assume that
the entries of the matrix $A$ are $g_{ij} a_{ij}$, where $a_{ij}$ are random variables
satisfying the assumptions of the proposition, and $g_{ij}$ are independent standard
normal random variables. We will write $\E_g$, $\P_g$ when we take expectations and
probability estimates with respect to $(g_{ij})$ (i.e. conditioned on $(a_{ij})$),
and we write $\E_a$ to denote the expectation with respect to $(a_{ij})$.

By Lemma~\ref{max tensor sums}, the random variable
$$
K^2 := \max_{j=1,\ldots,n} \Big\| \sum_{i=1}^N a_{ij}^2 B_i \otimes B_i \Big\|
$$
which does not depend on the random variables $(g_{ij})$, has expectation
\begin{equation}                        \label{exp K}
  \E_a (K^2) \le C (1 + a^2 b \sqrt{\log(2n)}).
\end{equation}
We condition on the random variables $(a_{ij})$; this fixes a value of $K$.

Let $X_1,\ldots,X_n \in \R^n$ denote the columns of the matrix $BA$; then
$$
X_j = \sum_{i=1}^N g_{ij} a_{ij} B_i, \qquad j=1,\ldots,n.
$$
Consider a $(1/2)$-net $\NN$ of the unit Euclidean sphere $S^{n-1}$ of cardinality
$|\NN| \le 5^n$, which exists by Lemma~\ref{net cardinality}.
Using Proposition~\ref{norm on nets}, we have
\begin{equation}                            \label{BA on nets}
  \|BA\|^2  = \|(BA)^*\|^2
  \le 4 \max_{x \in \NN} \|(BA)^* x\|_2^2
  =  4 \max_{x \in \NN} \sum_{j=1}^n \< X_j, x \> ^2.
\end{equation}

Fix $x \in \NN$. For every $j=1,\ldots,n$, the random variable
$$
\< X_j, x \> = \sum_{i=1}^N g_{ij} \< a_{ij} B_i, x\>
$$
is a Gaussian random variable with mean zero and variance
$$
\sum_{i=1}^N \< a_{ij} B_i, x\> ^2
\le \Big\| \sum_{i=1}^N a_{ij}^2 B_i \otimes B_i \Big\|
\le K^2.
$$
(To obtain the first inequality, take the supremum over $x \in S^{n-1}$).
Therefore, by Corollary~\ref{exponential deviation} with 
$d_i = (\Var \< X_i, x \> )^{1/2} \le K$, we have
for every $t > 0$:
$$
\P_g \Big\{ \Big(\sum_{j=1}^n \< X_j, x \> ^2 \Big)^{1/2} > K \sqrt{n} + t \Big\}
\le e^{-c_0 t^2 / K^2}.
$$
Let $s > 0$ be arbitrary. The previous estimate for $t = s K \sqrt{n}$ gives
$$
\P_g \Big\{ \Big(\sum_{j=1}^n \< X_j, x \> ^2 \Big)^{1/2} > (1+s) K \sqrt{n} \Big\}
\le e^{-c_0 s^2 n}.
$$
Taking the union bound over $x \in \NN$ and using \eqref{BA on nets}, we obtain
$$
\P_g \big\{ \|BA\| > 2 (1+ s) K \sqrt{n} \big\}
\le |\NN| e^{-c_0 s^2 n}
= 5^n e^{-c_0 s^2 n}
\le e^{(2 - c_0 s^2)n}.
$$
Integration yields
$$
\E_g \|BA\| \le C K \sqrt{n}.
$$
Finally, we take expectation with respect to the random variables $(a_{ij})$
and use \eqref{exp K} to conclude that
$$
\E \|BA\| \le C \E_a(K) \sqrt{n}
\le C_1 \big(1 + a^2 b \sqrt{\log(2n)} \big)^{1/2} \sqrt{n}.
$$
This completes the proof.
\end{proof}

\subsection{Theorem~\ref{main} for logarithmically small columns}

Our next step is to combine Propositions~\ref{up to log} and \ref{controlled columns}
and obtain a weaker version of the main Theorem~\ref{main} -- this time with the correct
bound $O(\sqrt{n})$ on the norm,
but under the additional assumption that the columns of the matrix $B$ are
logarithmically small.

\begin{theorem}                 \label{small columns}
  Let $\e \in (0,1)$ and let $N,n$ be positive integers.
  Consider an $N \times n$ random matrix $A$ whose entries are 
  independent random variables with mean zero and $(4+\e)$-th moment bounded by $1$.
  Let $B$ be an $n \times N$ matrix such that $\|B\| \le 1$, and whose columns satisfy
  for some $M \ge 1$ that 
  $$
  \|B_i\|_2 \le M \log^{-\frac{1}{2}-\frac{1}{\e}} (2n) \qquad \text{for every } i.
  $$
  Then
  $$
  \E \|BA\| \le C M^{1/2} \sqrt{n}.
  $$
\end{theorem}

\begin{proof}
By the symmetrization argument described in Section~\ref{s: prelims}, we can assume without loss
of generality that all entries $a_{ij}$ of the matrix $A = (a_{ij})$ are symmetric random variables.
Let
$$
a := \log^{\frac{1}{2\e}} (2n).
$$
We decompose every entry of the matrix $A$ according to its absolute value as
$$
\bar{a}_{ij} := a_{ij} \one_{\{ |a_{ij}| \le a \}}, \qquad
\tilde{a}_{ij} := a_{ij} \one_{\{ |a_{ij}| > a \}}.
$$
Then all random variables $\bar{a}_{ij}$ and $\tilde{a}_{ij}$ have mean zero,
and we have the following decomposition of matrices:
$$
BA = B \bar{A} + B \tilde{A}, \qquad
\text{where }
\bar{A} = (\bar{a}_{ij}), \;
\tilde{A} = (\tilde{a}_{ij}).
$$

The norm of $B \tilde{A}$ can be bounded using Proposition~\ref{up to log}.
Indeed, by the Truncation Lemma~\ref{truncation} with $p = 1 + \e/4$, we have
$$
\E \tilde{a}_{ij}^4
= \E a_{ij}^4 \one_{\{a_{ij}^4 > a^4\}}
\le \frac{\E a_{ij}^{4+\e}}{a^\e}
\le a^{-\e},
$$
where the last inequality follows from the moment hypothesis.
Therefore, the matrix $a^\e \tilde{A}$ satisfies the hypothesis of Proposition~\ref{up to log},
which then yields
$$
\E \|B \tilde{A}\| \le C a^{-\e} \sqrt{n \log(2n)} = C \sqrt{n}.
$$

The norm of $B \bar{A}$ can be bounded using Proposition~\ref{controlled columns},
which we can apply with $a$ as above and $b = M \log^{-\frac{1}{2}-\frac{1}{\e}}(2n)$. This gives
$$
\E \|B \bar{A}\| \le C (1 + a b^{1/2} \log^{1/4}(2n)) \sqrt{n}
\le 2 C M^{1/2} \sqrt{n},
$$
where the last inequality follows by our choice of $a$ and $b$.

Putting the two estimates together, we conclude by the triangle inequality that
$$
\E \|BA\| \le \E \|B \bar{A}\| + \E \|B \tilde{A}\|
\le C' M^{1/2} \sqrt{n}.
$$
This completes the proof.
\end{proof}

\begin{remark}
The factor $M^{1/2}$ in the conclusion of Theorem~\ref{small columns}
can easily be improved to about $M^{\e/2}$ by choosing 
$a = t \log^{\frac{1}{2\e}} (2n)$ in the proof and optimizing in $t$. 
We will not need this improvement in our argument.
\end{remark}

\section{Approach via concentration}                        \label{s: concentration}

In this section, we develop an alternative way to bound the norm of $BA$,
which rests on Gaussian concentration inequalities and elaborate
choice of $\e$-nets. 
The main technical result of this section is the following theorem, which, like
Theorem~\ref{small columns}, gives the correct bound $O(\sqrt{n})$ under
some boundedness assumptions on the entries of $A$. 

\begin{theorem}                         \label{small aij}
  Let $\e \in (0,1)$, $M \ge 1$ and let $N \ge n$ be positive integers such that $\log(2N) \le M n$.
  Consider an $N \times n$ random matrix $A$
  whose entries are independent random variables
  $a_{ij}$ with mean zero and such that
  $$
  \E |a_{ij}|^{2+\e} \le 1, \qquad
  |a_{ij}| \le \Big( \frac{Mn}{\log(2N)} \Big)^{\frac{1}{2+\e}} \qquad
  \text{for every } i,j.
  $$
  Let $B$ be an $n \times N$ matrix such that $\|B\| \le 1$.
  Then
  $$
  \E \|BA\| \le C(\e) \sqrt{M n}
  $$
  where $C(\e)$ depends only on $\e$.
\end{theorem}

\begin{remarks}

  {\bf 1.} If the entries $a_{ij}$ have bounded $(4+\e)$-th moment, it is easy to check that 
  $\max_{ij} a_{ij}\sim (nN)^\frac{1}{4+\e}$ holds with high probability. Therefore, 
  under the $(4+\e)$-th moment assumption, the hypotheses of 
  Theorem~\ref{small aij} are satisfied for almost square matrices, i.e. those for which 
  $N \le n^{1+c\e}$. This will quickly yield the main Theorem~\ref{main} for almost 
  square matrices, see Corollary~\ref{almost square} below.

  \medskip

  {\bf 2.} The hypotheses of Theorem~\ref{small aij} are almost sharp when $N \sim n$.  
  Indeed, let us assume for simplicity that the random variables $a_{ij}$ are identically distributed
  and $B$ is the identity matrix.
  The $(2+\e)$-th moment hypothesis is almost sharp: if $\E a_{ij}^2 \gg 1$ then 
  $(\E \|A\|^2)^{1/2} \ge \big( \frac{1}{n} \|A\|_\HS^2 \big)^{1/2} \gg \sqrt{n}$.
  Also, the boundedness hypothesis is almost sharp, since 
  $\|A\| \ge \max_{i,j} |a_{ij}|$.

  \medskip

  {\bf 3. } Using M.~Talagrand's concentration result, Theorem~\ref{talagrand}, 
  one can also obtains tail bounds for the norm $\|BA\|$:

\end{remarks}

\begin{corollary}							\label{small aij tail}
  Under the assumptions of Theorem~\ref{small aij}, one has for every $t > 0$:
  $$
  \P \big( \|BA\| > (C(\e) + t) \sqrt{M n} \big) \le 4e^{-t^2/4}.
  $$
  In particular, one has for every $q \ge 1$:
  $$
  (\E \|BA\|^{q})^{1/q} \le C_0(\e) \sqrt{q M n}.
  $$
\end{corollary}

\begin{proof}
We can consider the $N \times n$ matrix $A$ as a vector in $\R^{Nn}$. 
The Euclidean norm of such a vector equals the Hilbert-Schmidt norm $\|A\|_\HS$.
Since $\|BA\| \le \|B\| \|A\| \le 1 \cdot \|A\|_\HS$,
the function $f :\; \R^{Nn} \to \R$ defined by $f(A) = \|BA\|$ is $1$-Lipschitz
and convex. Since we have $|a_{ij}| \le \sqrt{Mn}$ for all $i,j$ by the assumptions,
M.~Talagrand's Theorem~\ref{talagrand} gives
$$
\P \big( \|BA\| - \E \|BA\| > t \sqrt{Mn} \big) \le 4 e^{-t^2/4}, \qquad t > 0.
$$
The estimate for $\E \|BA\|$ in Theorem~\ref{small aij} completes the proof.
\end{proof}

\subsection{Sparse matrices: rows and columns}

Theorem~\ref{small aij} will follow from our analysis of sparse matrices.
We will decompose the entries $a_{ij}$ according to their magnitude.
As the magnitude increases, the moment assumptions will ensure that there will
be fewer such entries, i.e. the resulting matrix becomes sparser.

We start with an elementary lemma, which will help us analyze
the magnitude of the rows and columns of the
matrix $BA$ when $A$ is a sparse matrix.

\begin{lemma}                       \label{rows cols}
  Let $N, n$ be positive integers. 
  Consider independent random variables $a_{ij}$, $i=1,\ldots,N$, $j=1,\ldots,n$.
  Let $p \in (0,1]$, and suppose that
  $$
  \E a_{ij}^2 \le p, \qquad
  |a_{ij}| \le 1 \qquad
  \text{for every } i,j.
  $$
  Let $B$ be an $n \times N$ matrix such that $\|B\| \le 1$, whose columns
  are denoted $B_i$. Then
  \begin{align}
    &\E \max_{i=1,\ldots,N} \sum_{j=1}^n a_{ij}^2 \le C(np + \log(2N)),      \label{max i} \\
    &\E \max_{j=1,\ldots,n} \sum_{i=1}^N a_{ij}^2 \|B_i\|_2^2 \le C(np + \log(2n)).    \label{max j}
  \end{align}
\end{lemma}

\begin{remark}
  The test case for this lemma, as well as for most of the results that follow, 
  is the random variables $a_{ij}$ with values in $\{-1,0,1\}$ and such that 
  $\P (a_{ij} \ne 0) = p$. The $N \times n$ random matrix $A = (a_{ij})$ will then become sparser as
  we decrease $p$; it will have on average $np$ nonzero entries per row. 
  Estimate \eqref{max i} gives a bound on the Euclidean norm of all rows of $A$.
\end{remark}

\begin{proof}
We will only prove inequality \eqref{max j}; the proof of inequality \eqref{max i} is similar.
By the assumptions, we have
$$
\Var(a_{ij}^2) \le \E a_{ij}^4 
\le \E a_{ij}^2
\le p \qquad
\text{for every } i,j.
$$
Also, recall that \eqref{B HS} gives
$$
\sum_{i=1}^N \|B_i\|_2^2 \le n, \qquad
\sum_{i=1}^N \|B_i\|_2^4 \le \max_i \|B_i\|_2^2 \cdot \sum_{i=1}^N \|B_i\|_2^2 \le n.
$$
Consider the sums of independent random variables
$$
S_j := \sum_{i=1}^N a_{ij}^2 \|B_i\|_2^2, \qquad
j=1,\ldots,n.
$$
The above estimates show that for every $j$ we have
$$
\E S_j = \sum_{i=1}^N \E(a_{ij}^2) \|B_i\|_2^2 \le np, \qquad
\Var(S_j) = \sum_{i=1}^N \Var(a_{ij}^2) \|B_i\|_2^4 \le np.
$$

We apply Bennett's inequality, Theorem~\ref{bennett}, for
$X_i = \frac{1}{2} \big( a_{ij}^2 - \E a_{ij}^2 \big) \|B_i\|_2^2$,
which clearly satisfy $|X_j| \le 1$ because $|a_{ij}| \le 1$
and $\|B_i\|_2 \le 1$ by \eqref{B HS}. We obtain
\begin{equation}                        \label{from bennett}
  \P \big\{ \frac{1}{2} (S_j - \E S_j) > t \big\}
   \le \exp \big( -\s^2 h(t/\s^2) \big)
\end{equation}
where $\E (\frac{1}{2}S_j) \le np$ and $\s^2 = \Var(\frac{1}{2}S_j) \le np$.
Note that $h(x) \ge cx$ for $x \ge 1$, where $c$ is some positive absolute constant.
Therefore, if $t \ge np$, then $\s^2 h(t/\s^2) \ge ct$, so
\eqref{from bennett} yields
$$
\P \{ S_j > 2t \} \le e^{-ct}  \qquad \text{for } t \ge np.
$$
Taking the union bound over all $j$, we conclude that
$$
\P \big\{ \max_{j=1,\ldots,n} S_j > 2t \big\}
\le n e^{-ct}  \qquad \text{for } t \ge np.
$$
Now let $s \ge 1$ be arbitrary, and use the last inequality for
$t = (np + \log(2n)) s$. We obtain
$$
\P \big\{ \max_{j=1,\ldots,n} S_j > 2(np + \log(2n)) s \big\}
\le n e^{- c \log(2n) s}
= 2^{-cs} n^{1-cs}.
$$
Integration yields
$$
\E \max_{j=1,\ldots,n} S_j \le C (np + \log(2n)).
$$
This completes the proof of \eqref{max j}.
\end{proof}

The estimates in Lemma~\ref{rows cols} motivate us to consider the class of
$N \times n$ matrices $A = (a_{ij})$ whose entries satisfy the following inequalities
for some parameters $p \in (0,1]$ and $K \ge 1$:
\begin{equation}                            \label{aij}
  \begin{aligned}
    &\max_{i,j} |a_{ij}| \le 1; \\
    &\max_{i=1,\ldots,N} \Big( \sum_{j=1}^n a_{ij}^2 \Big)^{1/2} \le K \sqrt{np + \log(2N)}; \\
    &\max_{j=1,\ldots,n} \Big( \sum_{i=1}^N a_{ij}^2 \|B_i\|_2^2 \Big)^{1/2} \le K \sqrt{np + \log(2n)}.
  \end{aligned}
\end{equation}
We have proved that for random matrices whose entries satisfy $|a_{ij}| \le 1$
and $E a_{ij}^2 \le p$, conditions \eqref{aij} hold 
with a random parameter $K$ that satisfies $\E K \le C$.

\subsection{Concentration for a fixed vector}

Our goal will be to estimate the magnitude of $\|BA\|$ for matrices of the form 
$A = (g_{ij} a_{ij})$,
where $g_{ij}$ are independent standard normal random variables, and
$a_{ij}$ are fixed numbers that satisfy conditions \eqref{aij}. Such an estimate
will be established in Proposition~\ref{norm sparse gauss matrices} below.
By the standard symmetrization, the same estimate will hold true if $A = (a_{ij})$
is a random matrix with entries as in Lemma~\ref{rows cols}.
This will be done in Corollary~\ref{norm sparse matrices}.
Finally, Theorem~\ref{small aij} will be deduced from this by decomposing the entries
of a random matrix according to their magnitude.

Our first step toward this goal is to check the magnitude of $\|BAx\|_2$ for a fixed
vector $x$.

\begin{lemma}                               \label{Ave}
  Let $N,n$ be positive integers. Consider an $N \times n$ random matrix 
  $A = (g_{ij} a_{ij})$
  where $g_{ij}$ are independent standard normal random variables and
  $a_{ij}$ are numbers that satisfy conditions \eqref{aij}.
  Let $B$ be an $n \times N$ matrix such that $\|B\| \le 1$.
  Then, for every vector $x \in B_2^n$ we have
  $$
  \E \|BAx\|_2 \le K \sqrt{np + \log(2n)}.
  $$
\end{lemma}

\begin{proof}
Denoting as usual the columns of $B$ by $B_i$, we have
$$
BAx = \sum_{i=1}^N \Big( \sum_{j=1}^n g_{ij} a_{ij} x_j \Big) B_i.
$$
Since $\|x\|_2 \le 1$ and using the last condition in \eqref{aij}, we have
\begin{align*}
  \E \|BAx\|_2^2
  &= \sum_{i=1}^N \sum_{j=1}^n a_{ij}^2 x_j^2 \|B_i\|_2^2 \\
  &= \sum_{j=1}^n \Big( \sum_{i=1}^N a_{ij}^2 \|B_i\|_2^2 \Big) x_j^2 \\
  &\le \max_{j=1,\ldots,n} \sum_{i=1}^N a_{ij}^2 \|B_i\|_2^2
  \le K^2 (np + \log(2n)).
\end{align*}
This completes the proof.
\end{proof}

We will now strengthen Lemma~\ref{Ave} into a deviation inequality for $\|BAx\|_2$.
This is a simple consequence of the Gaussian concentration, Theorem~\ref{gauss concentration}.
This deviation inequality is universal in that it holds for any vector $x$; in the sequel we will
need more delicate inequalities that depend on the distribution of the coordinates in $x$.

\begin{lemma}[Universal deviation]              \label{universal deviation}
  Let $A$ and $B$ be matrices as in Lemma~\ref{Ave}.
  Then, for every vector $x \in B_2^n$ and every $t > 0$ we have
  \begin{equation}                      \label{eq universal deviation}
    \P \big\{ \|BAx\|_2 > K \sqrt{np + \log(2n)} + t \big\} \le e^{-c_0 t^2}.
  \end{equation}
\end{lemma}

\begin{proof}
As in the proof of Lemma~\ref{Ave}, we write
$$
BAx = \sum_{i=1}^N \Big( \sum_{j=1}^n g_{ij} a_{ij} x_j \Big) B_i
$$
where $B_i$ are the columns of the matrix $B$.
Therefore, the random vector $BAx$ is distributed identically with the random vector
$$
\sum_{i=1}^N g_i \l_i B_i,  \qquad
\text{where }
\l_i = \Big( \sum_{j=1}^n a_{ij}^2 x_j^2 \Big)^{1/2}
$$
and where $g_i$ are independent standard normal random variables.
Since all $|a_{ij}| \le 1$  by conditions \eqref{aij}, and $\|x\|_2 \le 1$  by the assumptions,
we have
$$
0 \le \l_i \le 1, \qquad i=1,\ldots,N.
$$
Consider the map $f : \R^N \to \R$ given by
$$
f(y) = \Big\| \sum_{i=1}^N y_i \l_i B_i \Big\|_2.
$$
Its Lipschitz norm equals
$$
\|f\|_\Lip
= \Big\| \sum_{i=1}^N \l_i^2 B_i \otimes B_i \Big\|^{1/2}
\le \max_i |\l_i| \cdot \Big\| \sum_{i=1}^N B_i \otimes B_i \Big\|^{1/2}
\le 1 \cdot \|B\| \le 1.
$$
Then the Gaussian concentration, Theorem~\ref{gauss concentration}, gives for every $t >0$:
$$
\P ( f(g) - \E f(g) > t ) \le \exp(-c_0 t^2),
$$
where $g = (g_1,\ldots,g_N)$.
Since as we noted above, $f(g)$ is distributed identically with $\|BAx\|_2$,
Lemma~\ref{Ave} completes the proof.
\end{proof}

\subsection{Control of sparse vectors}

Since the spectral norm of $BA$ is the supremum of $\|BAx\|_2$ over all $x \in S^{n-1}$,
the result of Lemma~\ref{universal deviation} suggests that
$\E \|BA\| \lesssim \sqrt{np + \log N}$ should be true.
However, the deviation inequality in Lemma~\ref{universal deviation} is not
strong enough to prove this bound.
This is because the metric entropy of the sphere,
measured e.g. as the cardinality of its $\frac{1}{2}$-net, is $e^{cn}$. If we are to make the
bound on $\|BAx\|_2$ uniform over the net, we would need the probability estimate in
\eqref{eq universal deviation} at most $e^{-cn}$ (to allow a room for the union bound over $e^{cn}$
points $x$ in the net). This however would force us to make $t \sim \sqrt{n}$ or larger,
so the best bound we can get this way is $\E \|BA\|_2 \lesssim \sqrt{n}$.
This bound is too weak as it ignores the last two assumptions in \eqref{aij}.

Nevertheless, the bound in Lemma~\ref{universal deviation} can be made
uniform over a set of sparse vectors, whose metric entropy is smaller
than that of the whole sphere:

\begin{proposition}[Sparse vectors]                 \label{sparse}
  Let $A$ and $B$ be matrices as in Lemma~\ref{Ave}.
  There exists an absolute constant $c>0$ such that the following holds.
  Consider the set of vectors
  $$
  B_{2,0} := \Big\{ x \in \R^n, \; \|x\|_2 \le 1, \; \|x\|_0 \le cnp/\log(e/p) \Big\}.
  $$
  Then
  $$
  \E \sup_{x \in B_{2,0}} \|BAx\|_2 \le 3K\sqrt{np + \log(2n)}.
  $$
\end{proposition}

\begin{proof}
Let $c>0$ be a constant to be determined later, and let
$\l := cp/\log(e/p)$.
Then
$$
B_{2,0} = \bigcup_{|J| = \lfloor \l n \rfloor} B_2^J,
$$
where the union is over all subsets $J \subset \{1,\ldots,n\}$
of cardinality $ \lfloor \l n \rfloor$, and where
$B_2^J = \{ x \in \R^J: \; \|x\|_2 \le 1\}$ denotes the unit Euclidean ball in $\R^J$.
By Lemma~\ref{net cardinality}, $B_2^J$ has a $\frac{1}{2}$-net $\NN_J$ of cardinality
at most $e^{2 \l n}$.
Let $t \ge 1$. For a fixed $x \in \NN_J$, Lemma~\ref{universal deviation} gives
$$
\P \big\{ \|BAx\|_2 > (K+1)\sqrt{np + \log(2n)} + t \big\}
\le \exp \big( -c_0 (np + t^2) \big).
$$
Using Proposition~\ref{norm on nets} and
taking the union bound over all $x \in \NN_J$, we obtain
\begin{align*}
  \P &\big\{ \frac{1}{2} \sup_{x \in B_2^J} \|BAx\|_2 > (K+1)\sqrt{np + \log(2n)} + t \big\} \\
  &\le \P \big\{ \sup_{x \in \NN_J} \|BAx\|_2 > (K+1)\sqrt{np + \log(2n)} + t \big\} \\
  &\le |\NN_J| \exp \big( -c_0 (np + t^2) \big)
  \le \exp \big( 2 \l n - c_0 (np + t^2) \big).
\end{align*}
Since there are $\binom{n}{\lfloor \l n \rfloor} \le (e/\l)^{\l n}$ ways to choose the subset $J$,
by taking the union bound over all $J$ we conclude that
\begin{multline}                            \label{exp long}
  \P \big\{ \frac{1}{2} \sup_{x \in B_{2,0}} \|BAx\|_2 > 2(K+1)\sqrt{np + \log(2n)} + t \big\} \\
  \le \exp \big( \l \log(e/\l) n + 2 \l n - c_0 (np + t^2) \big).
\end{multline}
Finally, if the absolute constant $c>0$ in the definition of $\l$ is chosen sufficiently
small, we have $\l \log(e/\l) n + 2 \l n \le c_0 n p$. Thus the right hand side of \eqref{exp long}
is at most
$$
\exp(-c_0 t^2).
$$
Integration completes the proof.
\end{proof}

\subsection{Control of spread vectors}

Although we now have a good control of sparse vectors,
they unfortunately comprise a small part of the unit ball $B_2^n$.
More common but harder to deal with are ``spread vectors''
-- those having many coordinates that are not close to zero.
The next result gains control of the spread vectors.

\begin{proposition}[Spread vectors]                 \label{spread}
  Let $A$ and $B$ be matrices as in Lemma~\ref{Ave} with $N \ge n$.
  Let $M \ge 2$.
  Consider the set of vectors
  $$
  B_{2,\infty} := \Big\{ x \in \R^n, \; \|x\|_2 \le 1, \; \|x\|_\infty \le \frac{M}{\sqrt{n}} \Big\}.
  $$
  Then
  $$
  \E \sup_{x \in B_{2,\infty}} \|BAx\|_2 \le C \log^{3/2}(M) \cdot K \sqrt{np + \log(2N)}.
  $$
\end{proposition}

\begin{proof}
This time we will need to work with multiple nets to account for
different possible distributions of the magnitude of the coordinates
of vectors $x \in B_{2,\infty}$. Since $\|x\|_\infty \le \|x\|_2$, without loss
of generality we can assume that $M \le \sqrt{n}$.

\smallskip

{\em Step 1: construction of nets.}
Let
$$
h_k := \frac{2^k}{\sqrt{n}}, \qquad
k = -2, -1, 0, 1, 2, \ldots, \log_2 M
$$
and let
$$
\NN := \{ x \in B_{2,\infty} :\; \forall j \; \exists k \text{ such that } |x_j| = h_k\}.
$$
A standard calculation shows that $\NN$ is an $\frac{1}{2}$-net of $B_{2,\infty}$
in the $B_{2,\infty}$-norm, i.e. for every $x \in B_{2,\infty}$ there exists $y \in \NN$
such that $x-y \in \frac{1}{2} B_{2,\infty}$.
Therefore, by Proposition~\ref{norm on nets},
$$
\sup_{x \in B_{2,\infty}} \|BAx\|_2
\le 2 \sup_{x \in \NN} \|BAx\|_2.
$$

Fix $x \in \NN$. 
Since $\|x\|_2 \le 1$, the number of coordinates of $x$ that satisfy
$|x_j| = h_k$ is at most $\lfloor h_k^{-2} \rfloor$, for every $k$.
Decomposing $x$ according to the coordinates whose absolute value is $h_k$,
we have by the triangle inequality
that
\begin{equation}                                        \label{sum log M}
  \sup_{x \in B_{2,\infty}} \|BAx\|_2
  \le 2 \sum_{k=-2}^{\log_2 M} \sup_{x \in \NN_k} \|BAy\|_2,
\end{equation}
where
$$
\NN_k = \big\{ x \in B_2^n :\;
\|x\|_0 \le \lfloor h_k^{-2} \rfloor;
\text{ all nonzero coordinates of $x$ satisfy } |x_j| = h_k
\big\}.
$$

Fix $k$ and assume that $\NN_k \ne \emptyset$. Since $h_k \le M/\sqrt{n}$, we have 
\begin{equation}							\label{m}
  m := \lfloor h_k^{-2} \rfloor 
  \ge \lfloor n/M^2 \rfloor 
  \ge 1.
\end{equation}
To estimate the cardinality of $\NN_k$, note that there are at most
$\min(m,n)$ ways to choose $\|x\|_0 := l$; there are $\binom{n}{l}$
ways to choose the support of $x$; and there are $2^l$
ways to choose the (signs of) nonzero coordinates of $x$.
Hence by Stirling's approximation and using \eqref{m}, we have
\begin{equation}                                    \label{Nk cardinality}
  |\NN_k|
  \le \sum_{l=1}^{\min(m,n)} \binom{n}{l} 2^l
  \le \min \big\{ \Big( \frac{2en}{m} \Big)^m, 4^n \big\}
  \le (4eM^2)^m
  \le \exp(C m \log M)
\end{equation}
where $C \ge 1$ is an absolute constant.

\smallskip

{\em Step 2: control of a fixed vector.}
Fix $m$ and fix $x \in \NN_k$.
As we saw in the proof of Lemma~\ref{universal deviation},
$$
\|BAx\|_2 \text{ is distributed identically with } \Big\| \sum_{i=1}^N g_i \l_i B_i \Big\|_2
$$
where
$$
\l_i = \Big( \sum_{j=1}^n a_{ij}^2 x_j^2 \Big)^{1/2}
$$
and where $g_i$ are independent standard normal random variables.
Since $x \in \NN_k$, we have $\|x\|_\infty = h_k \le \frac{1}{\sqrt{m}}$.
This and the second condition in \eqref{aij} yield
$$
\l_i \le \Big( \frac{1}{m} \sum_{j=1}^n a_{ij}^2 \Big)^{1/2}
\le K \sqrt{ \frac{np + \log(2N)}{m} }.
$$
We consider the map $f : \R^N \to \R$ given by
$$
f(y) = \Big\| \sum_{i=1}^N y_i \l_i B_i \Big\|_2.
$$
Repeating the estimate in the proof of Lemma~\ref{universal deviation},
we bound the Lipschitz norm as
$$
\|f\|_\Lip
\le \max_i |\l_i|
\le K \sqrt{ \frac{np + \log(2N)}{m} }.
$$
Then the Gaussian concentration, Theorem~\ref{gauss concentration}, gives for every $t >0$:
$$
\P ( f(g) - \E f(g) > t )
\le \exp \Big( -\frac{c_0 t^2 m}{K^2 (np + \log(2N))} \Big),
$$
where $g = (g_1,\ldots,g_N)$.
Since as we noted above, $f(g)$ is distributed identically with $\|BAx\|_2$,
Lemma~\ref{Ave} yields that
$$
\P ( \|BAx\|_2 > K \sqrt{np + \log(2n)} + t )
\le \exp \Big( -\frac{c_0 t^2 m}{K^2 (np + \log(2N))} \Big),
$$
Let $u > 0$ be arbitrary. Applying the above estimate for
$t = u K \sqrt{np + \log(2N)}$ and using $N \ge n$ we conclude that
\begin{equation}                                \label{fixed x}
  \P \big( \|BAx\|_2 > (1+u) K \sqrt{np + \log(2N)} \big)
  \le \exp(-c_0 u^2 m).
\end{equation}

\smallskip

{\em Step 3: union bound.}
Taking the union bound in \eqref{fixed x} over all $x \in \NN_k$
and using estimate \eqref{Nk cardinality} on the cardinality of $\NN_k$, we have
for all $u > 0$:
\begin{align*}
\P \big( \sup_{x \in \NN_k} \|BAx\|_2 > (1+u) K \sqrt{np + \log(2N)} \big)
  &\le |\NN_k| \exp(-c_0 u^2 m) \\
  &\le \exp(C m \log M - c_0 u^2 m).
\end{align*}
Let $s \ge 1$. We choose $u = C_1 s \sqrt{\log M}$, where $C_1 := \sqrt{C/c_0}$.
Since $u \ge 1$ and $m \ge 1$, $M \ge 2$, we obtain from the above estimate that
\begin{align*}
\P \big( \sup_{x \in \NN_k} \|BAx\|_2 > 2 C_1 s K \sqrt{\log(M) (np + \log(2N))}  \big)
  &\le \exp(C (1-s^2) m \log M) \\
  &\le \exp(c(1-s^2)).
\end{align*}
Integrating yields that 
$$
\E \sup_{x \in \NN_k} \|BAx\|_2 \le C_2 K \sqrt{\log(M) (np + \log(2N))}.
$$
Putting this back in \eqref{sum log M}, we conclude that 
$$
\E \sup_{x \in B_{2,\infty}} \|BAx\|_2
\le 2 (3 + \log M) \cdot C_2 K \sqrt{\log(M) (np + \log(2N))}.
$$
This completes the proof.
\end{proof}

\subsection{Norms of sparse matrices, and proof of Theorem~\ref{small aij}}

Propositions~\ref{sparse} and \ref{spread} together handle all vectors in the 
unit ball, and yield the following norm estimate:

\begin{proposition}						\label{norm sparse gauss matrices} 
  Let $A$ and $B$ be matrices as in Lemma~\ref{Ave} with $N \ge n$.
  Then 
  $$
  \E \|BA\| \le C \log^{3/2} \Big( \frac{e}{p} \Big) \cdot K \sqrt{np + \log(2N)}.
  $$
\end{proposition}

\begin{proof}
Let $c$ be the absolute constant as in Proposition~\ref{sparse}; 
we can clearly assume that $c \le 1/4$.
We define
$$
M = \sqrt{ \frac{1}{cp} \log \frac{e}{p} }.
$$
Note that $M \ge 2$ as required in Proposition~\ref{sparse}.

Fix a vector $x \in B_2^n$. We decompose it according to the magnitude of the coordinates, as follows: 
$$
x = y + z, \qquad
y := x \, \one_{\{j : \; |x_j| > M/\sqrt{n} \}}, \qquad
z := x \, \one_{\{j : \; |x_j| \le M/\sqrt{n} \}}.
$$
Clearly, $\|y\|_2 \le \|x\|_2 \le 1$, $\|z\|_2 \le \|x\|_2 \le 1$.
By Markov's inequality, we have
$$
\|y\|_0 = \big| \{j : \; |x_j| > M/\sqrt{n} \} \big|
\le \frac{n}{M^2}
= \frac{cnp}{\log(e/p)}.
$$
Then $y \in B_{2,0}$ as in Proposition~\ref{sparse}.
On the other hand, $\|z\|_\infty \le M/\sqrt{n}$ by definition, so 
$z \in B_{2,\infty}$ as in Proposition~\ref{spread}.
Therefore, by Propositions~\ref{sparse} and \ref{spread} we have
\begin{align*}
\E \|BA\|
  &= \E \sup_{x \in B_2^n} \|BAx\|_2 
  \le \E \sup_{y \in B_{2,0}} \|BAy\|_2 
    + \E \sup_{z \in B_{2,\infty}} \|BAz\|_2 \\
  &\le 3K\sqrt{np + \log(2n)} 
    + C \log^{3/2}(M) \cdot K \sqrt{np + \log(2N)}.
\end{align*}
Our choice of $M$ and the assumption $N \ge n$ completes the proof.
\end{proof}

Finally, a standard symmetrization argument yields the following norm estimate, 
which we shall use for sparse random matrices. 

\begin{corollary}						\label{norm sparse matrices}
  Let $p \in (0,1]$ and let $N \ge n$ be positive integers. 
  Consider an $N \times n$ random matrix $A$ whose entries are independent random variables
  $a_{ij}$ with mean zero and such that
  $$
  \E |a_{ij}|^2 \le p, \qquad
  |a_{ij}| \le 1 \qquad
  \text{for every } i,j.
  $$
  Let $B$ be an $n \times N$ matrix such that $\|B\| \le 1$.
  Then
  $$
  \E \|BA\| \le C \log^{3/2} \Big( \frac{e}{p} \Big) \sqrt{np + \log(2N)}.
  $$
\end{corollary}

\begin{remark}
  It would be interesting to remove the logarithmic term from this estimate.
\end{remark}

\begin{proof}
Let $g_{ij}$ be independent standard normal random variables. Consider the random matrix
$\tilde{A} = (g_{ij}a_{ij})$. By \eqref{making gaussian}, we have
\begin{equation}							\label{B tilde A}
 \E \|BA\| \le  (2\pi)^{1/2} \, \E \|B \tilde{A}\|.
\end{equation}
By Lemma~\ref{rows cols}, conditions \eqref{aij} hold with some random parameter $K \ge 1$
which only depends on the random variables $(a_{ij})$ and not on $(g_{ij})$, and 
which satisfies 
\begin{equation}							\label{EK}
  \E_a K \le C_1
\end{equation}
where $C_1$ is an absolute constant.
Here and below we write $\E_a$ when the expectation is with respect to $(a_{ij})$, 
and $\E_g$ if the expectation is with respect to $(g_{ij})$.

Condition on the random variables $(a_{ij})$. Proposition~\ref{norm sparse gauss matrices} 
then yields
$$
\E_g \|B \tilde{A}\| \le C \log^{3/2} \Big( \frac{e}{p} \Big) \cdot K \sqrt{np + \log(2N)}.
$$
Therefore, when we remove the conditioning, we obtain by \eqref{EK} that
$$
\E \|B \tilde{A}\| = \E_a \E_g \|B \tilde{A}\| 
\le C \log^{3/2} \Big( \frac{e}{p} \Big) \cdot C_1 \sqrt{np + \log(2N)}.
$$
This and \eqref{B tilde A} complete the proof.
\end{proof}

\bigskip

\begin{proof}[Proof of Theorem~\ref{small aij}.]
By the standard symmetrization technique described in Section~\ref{s: prelims}, 
we can assume without loss of generality that all $a_{ij}$ are symmetric random variables. 
We decompose the matrix $A$ according to the magnitude of its entries as follows. 
Given a subset $I \subset \R$, we define the truncated matrix 
$$
\trunc(A, I) = (a_{ij} \one_{\{ |a_{ij}| \in I \}}).
$$
Consider 
\begin{align*}
  &A^{(0)} = \trunc(A, [0,1]); \\
  &A^{(k)} = 2^{-k} \trunc(A, (2^{k-1}, 2^k] ), \quad k=1,2,\ldots
\end{align*}
Then we have a decomposition
$A = \sum_{k=0}^\infty 2^k A^{(k)}$.
This sum is actually finite because of the boundedness assumption on $a_{ij}$.
Indeed, we have
\begin{equation}							\label{A decomposed}
  A = A^{(0)} + \sum_{k=1}^{k_0} 2^k A^{(k)}
\end{equation}
where $k_0$ is the maximal integer such that 
\begin{equation}							\label{k0}
  2^{k_0-1} \le \Big( \frac{Mn}{\log(2N)} \Big)^{\frac{1}{2+\e}}.
\end{equation}
Because $a_{ij}$ are symmetric random variables, all entries $a_{ij}^{(k)}$ of the matrices $A^{(k)}$ 
satisfy $\E a_{ij}^{(k)} = 0$ and $|a_{ij}^{(k)}| \le 1$. 

Using Corollary~\ref{norm sparse matrices} for the matrix $A^{(0)}$ and $p=1$, we obtain
\begin{equation}							\label{BA0}
  \E \|BA^{(0)}\| \le C_1 \sqrt{n + \log(2N)}
  \le 2C_1 \sqrt{Mn},
\end{equation}
where the last line follows because $\log(2N) \le Mn$ and $M \ge 1$ by the hypothesis.

Now we fix $1 \le k \le k_0$. Using the $(2+\e)$-th moment
assumption, we have by Markov's inequality that
$$
\P (a_{ij}^{(k)} \ne 0)
\le \P (a_{ij} > 2^{k-1})
\le 2^{-(2+\e)(k-1)} =: p_k.
$$
This and the bound $|a_{ij}^{(k)}| \le 1$ yield
$\E (a_{ij}^{(k)})^2 \le p_k.$
With this, we apply Corollary~\ref{norm sparse matrices} for the matrix $A^{(k)}$
and obtain
$$
\E \|BA^{(k)}\| \le C \log^{3/2} \Big( \frac{e}{p_k} \Big) \sqrt{n p_k + \log(2N)}.
$$
By the definition of $p_k$ and by \eqref{k0}, we have
$$
p_k \ge p_{k_0} \ge \frac{\log(2N)}{Mn}.
$$
Therefore, $n p_k + \log(2N) \le (1+M) n p_k \le 2Mnp_k$, so 
\begin{align}										\label{BAk}
\E \|BA^{(k)}\| 
  &\le C \log^{3/2} \Big( \frac{e}{p_k} \Big) \sqrt{2Mn p_k} \nonumber\\
  &\le C_2 \big[ 1 + (2+\e)(k-1) \big]^{3/2} 2^{-(1+\e/2)(k-1)} \cdot \sqrt{Mn}.
\end{align}

Using \eqref{A decomposed} and the triangle inequality, then 
using \eqref{BA0} and \eqref{BAk}, we conclude that 
\begin{align*}
\E \|BA\| 
  &\le \E \|BA^{(0)}\| + \sum_{k=1}^{k_0} 2^k \, \E \|BA^{(k)}\| \\
  &\le 2 C_1 \sqrt{Mn} 
    + \sum_{k=1}^{k_0} C_2 \big[ 1 + (2+\e)(k-1) \big]^{3/2} 2^{k-(1+\e/2)(k-1)} \cdot \sqrt{Mn} \\
  &\le C_2 \sqrt{Mn} \cdot \sum_{k=1}^\infty k^{3/2} 2^{-(\e/2)k} \\
  &= C(\e) \sqrt{Mn}.
\end{align*}
This completes the proof of Theorem~\ref{small aij}.
\end{proof}

\subsection{Almost square matrices}

The main application of Theorem~\ref{small aij} is for almost square 
matrices -- those for which $N = n^{1+o(1)}$. The next lemma verifies the hypotheses of
Theorem~\ref{small aij} for such matrices. 

\begin{lemma}								\label{almost square hypotheses}
  Let $\e \in (0,1)$ and let $N,n$ be positive integers satisfying $N \le n^{1 + \e/10}$.
  Let $A$ be an $N \times n$ random matrix whose entries are independent random variables
  with $(4+\e)$-th moment bounded by $1$.
  Define the random variable $M$ by the equation
  \begin{equation}							\label{M}
    \max_{i,j} |a_{ij}| = \Big( \frac{Mn}{\log(2N)} \Big)^{\frac{1}{2+\e/4}}.
  \end{equation}
  Then, for every $t \ge 1$, one has
  $$
  \P \big( M > C(\e) t \big) \le \frac{1}{t^2}.
  $$
  In particular, one has  $\E M \le C_1(\e)$.
\end{lemma}

\begin{proof}
By Markov's inequality, we have for every $i,j$ that
$$
\P \big( |a_{ij}| > s \big)
\le \frac{1}{s^{4+\e}}, \qquad s > 0.
$$
Let $t \ge 1$. We then have 
$$
\P \big( |a_{ij}| > (t^2 n N)^{\frac{1}{4+\e}} \big)
\le \frac{1}{t^2 n N}.
$$
Taking the union bound over all $nN$ random variables $a_{ij}$, we obtain 
\begin{equation}								\label{max aij tail}
  \P \big( \max_{i,j} |a_{ij}| > (t^2 n N)^{\frac{1}{4+\e}} \big)
  \le \frac{1}{t^2}.
\end{equation}
The assumption $N \le n^{1+\e/10}$ yields that 
$$
nN \le \Big( \frac{C(\e)n}{\log(2N)} \Big)^{2+\e/8}.
$$
Therefore, since $\frac{2+\e/8}{4+\e} \le \frac{1}{2+\e/4}$ and $t \ge 1$, we have
$$
(t^2 n N)^\frac{1}{4+\e} \le \Big( \frac{C(\e) t n}{\log(2N)} \Big)^\frac{1}{2+\e/4}.
$$
Using this in \eqref{max aij tail}, we obtain
$$
\P \big( M > C(\e) t \big)
\le \P \Big( \max_{i,j} |a_{ij}| > \Big( \frac{C(\e) t n}{\log(2N)} \Big)^\frac{1}{2+\e/4} \Big)
\le \frac{1}{t^2}.
$$
Integration completes the proof.
\end{proof}

We are now ready to state and prove a partial case of Theorem~\ref{main}
for almost square matrices.

\begin{corollary}							\label{almost square}
  Let $\e \in (0,1)$ and let $N,n$ be positive integers satisfying $N \le n^{1 + \e/10}$.
  Let $A$ be an $N \times n$ random matrix whose entries are independent random variables
  with mean zero and $(4+\e)$-th moment bounded by $1$.
  Let $B$ be an $n \times N$ matrix such that $\|B\| \le 1$.
  Then
  $$
  \E \|BA\| \le C(\e) \sqrt{n}.
  $$
\end{corollary}

\begin{proof}
Without loss of generality we may assume that $N \ge n$ by adding an appropriate number
of zero rows to $A$ and zero columns to $B$.
Also, using the standard symmetrization, we can assume that the random variables $a_{ij}$ are
symmetric. 
Let $M$ be the random variable as in Lemma~\ref{almost square hypotheses}, 
and let $t \ge 1$. By the definition, $\{ M \le t \}$ is the product event. Therefore, conditioning on this event
(i) preserves the independence of the entries of $A$; 
(ii) makes all these entries bounded as in \eqref{M};
(iii) can only reduce their moments by Lemma~\ref{exp reduces by conditioning}, thus for all $i,j$ we have
$$
\E \big( |a_{ij}|^{2+\e/4} \,|\, M \le t \big) \le \E |a_{ij}|^{2+\e/4} \le 1.
$$
Therefore, we can apply Corollary~\ref{small aij tail} conditionally, with $\e/4$ and with $M$ replaced
by $\max(M,10)$, which gives
$$
\big[ \E \big( \|BA\|^2 \,|\, M \le t \big) \big]^{1/2} \le C_0(\e) \sqrt{tn}
\quad \text{for } t \ge 1.
$$
Additionally, by Lemma~\ref{almost square hypotheses} we have
$$
\P \big( M > C(\e) t \big) \le \frac{1}{t^2}
\quad \text{for } t \ge 1.
$$
By Lemma~\ref{conditional exp}, this yields
$$
\E \|BA\| 
\le (\E \|BA\|^2)^{1/2}
\le C_1(\e) \sqrt{n}
$$
as claimed.
\end{proof}

\section{Completion of the proof of Theorem~\ref{main}}

\begin{proof}[Proof of Theorem~\ref{main}.]

By adding an appropriate number of zero rows to $B$ or zero columns to $A$ we can 
assume that $m = n$, thus $B$ is an $n \times N$ matrix. Consider the exponent
$$
K = K(\e) = \frac{1}{2} + \frac{1}{\e}.
$$
As usual, let $B_1,\ldots,B_N$ be the columns of the matrix $B$. 
Consider the subset $I \subset \{1,\ldots, N\}$ of large columns defined as 
$$
I := \big\{ i :\; \|B_i\|_2 > C_0(\e) \log^{-K} (2n) \big\}.
$$
Here we choose $C_0(\e)$ sufficiently large so that,
by \eqref{B HS} and Markov's inequality, we have 
$$
N_0 := |I| < C_0(\e)^{-2} n \log^{2K} (2n) \le n^{1+\e/10}.
$$
Denote by $A_I$ the $N_0 \times n$ submatrix of $A$ whose rows are in $I$, 
by $B_I$ the $n \times N_0$ submatrix of $B$ whose columns are in $I$ 
(and similarly for $I^c$). The decomposition $BA = B_I A_I + B_{I^c} A_{I^c}$
implies by the triangle inequality that 
\begin{equation}						\label{BA split}
  \|BA\| \le \|B_I A_I\| + \|B_{I^c} A_{I^c}\|.
\end{equation}
This splits our problem into two subproblems, one for $I$ and one for $I^c$.
Of course, if $I$ or $I^c$ is empty then the corresponding matrix is zero 
and we can skip its estimation. 

The matrices $A_I$, $B_I$ are almost square, so Corollary~\ref{almost square} 
applies for them, giving
\begin{equation}						\label{for I}
  \E \|B_I A_I\| \le C(\e) \sqrt{n}.
\end{equation}

On the other hand, the columns of the matrix $B_{I^c}$ are small by the definition of $I$:
$$
\|B_i\|_2 \le C_0(\e) \log^{-K} (2n) \qquad \text{for every } i \in I^c.
$$
Therefore, Theorem~\ref{small columns} applies to the matrices $A_{I^c}$, $B_{I^c}$,
which gives 
\begin{equation}						\label{for Ic}
  \E \|B_{I^c} A_{I^c}\| \le C_1(\e) \sqrt{n}.
\end{equation}
Putting estimates \eqref{for I} and \eqref{for Ic} into \eqref{BA split}, we conclude that
$$
\E \|BA\| \le C_2(\e) \sqrt{n}.
$$
Theorem~\ref{main} is proved.
\end{proof}

\end{document}